\documentclass[12pt]{article}

\usepackage{amsmath, amsthm, amssymb}
\usepackage{mathptmx} 

\usepackage{graphicx}
\usepackage{float}
\usepackage{multirow}
\usepackage{tikz}

\usepackage{algorithm}
\usepackage{algpseudocode} 

\usepackage{subcaption}

\usepackage[a4paper]{geometry}
\usepackage{xcolor}
\usepackage{hyperref} 

\newtheorem{theorem}{Theorem}[section]

\newtheorem{Remark}{Remark}[section]
\newtheorem{Lemma}{Lemma}[section]
\newtheorem{Corollary}{Corollary}
\newtheorem{Definition}{Definition}

\normalsize

\newcommand{\bd}{\begin{displaymath}}
\newcommand{\ed}{\end{displaymath}}
\newcommand{\be}{\begin{equation}}
\newcommand{\ee}{\end{equation}}
\newcommand{\bea}{\begin{eqnarray}}
\newcommand{\eea}{\end{eqnarray}}
\newcommand{\bda}{\begin{eqnarray*}}
\newcommand{\eda}{\end{eqnarray*}}
\newcommand{\ba}{\begin{array}}
\newcommand{\ea}{\end{array}}

\begin{document}
	
\title{Moreau Envelope-Based Clustering for Generalized Multi-Source Weber Problem}

\author{
Nguyen Thi Thu Van\thanks{Department of Mathematics and Statistics, University of Economics Ho Chi Minh City, Ho Chi Minh City, Vietnam. Email: van.nguyen@ueh.edu.vn}}

\date{\today}

\maketitle

\begin{abstract}
In this paper, we propose an efficient algorithm for data clustering based on the
Moreau envelope, which approximates nonsmooth and nonconvex components of the generalized multi-source Weber problem. 
The number of clusters is not fixed in advance and is determined automatically by progressively removing empty or redundant clusters. 
The smoothing induced by the Moreau envelope transforms the original problem into a structured optimization task that can be efficiently solved using first-order methods and simple matrix vector operations. 
Numerical experiments on synthetic and real datasets show that the proposed approach is fast, scalable, and competitive with existing methods in both clustering quality and computational efficiency.
\end{abstract}

\noindent\textbf{Keywords:} Data clustering, Moreau envelope, nonsmooth and nonconvex optimization, DC programming, variational analysis.

\section{Introduction}

Let
$A=\{a^1,\ldots,a^m\}\subset \mathbb{R}^n$
be a finite set of demand points.
The generalized multi-source Weber problem (GMWP) concerns the location of
$k$ facilities to minimize the sum of distances from each demand point
to its nearest facility.
It can be formulated as
\begin{equation}\label{maxminn}
\min_{x=(x^1,\ldots,x^k)\in\mathbb{R}^{nk}}
f_F(x)
:=
\sum_{i=1}^{m}
\min_{\ell=1,\ldots,k}
\rho_F\bigl(x^\ell-a^i\bigr),
\end{equation}
where $\rho_F$ denotes the Minkowski functional (or gauge)
associated with a nonempty compact convex set
$F\subset\mathbb{R}^n$ satisfying $0\in\mathrm{int}(F)$, defined by
\[
\rho_F(x)
=
\inf\bigl\{t\ge 0 \mid x\in tF\bigr\},
\qquad x\in\mathbb{R}^n.
\]

Problem~\eqref{maxminn} unifies and extends several classical models in
facility location theory, including the Fermat--Torricelli problem and the
single-source Weber problem; see
\cite{MordukhovichNam11, MordukhovichNam23}.
It also serves as a flexible framework for clustering and multi-facility
location models, which have been extensively studied using tools from
variational analysis and difference-of-convex (DC) programming; see, e.g.,
\cite{CuongThienYaoYen23, CuongWenYaoYen24, CuongYaoYen20, LongNamSharkanskyYen25, NamRectorGiles17, NamTranReynolds18}.

Due to the presence of the minimum operator and the induced assignment
structure, problem~\eqref{maxminn} is intrinsically nonsmooth and nonconvex.
These features pose significant challenges for both theoretical analysis
and numerical computation, particularly in large-scale settings.

It was shown in~\cite{LongNamTranVan24} that the set of global minimizers of
problem~\eqref{maxminn} is nonempty and compact.
Moreover, the authors derived an explicit DC decomposition of the objective
function, which provides a convenient variational framework for stationarity
analysis and algorithm design.
Specifically, for any
$x=(x^1,\ldots,x^k)\in\mathbb{R}^{nk}$, one has
\[
f_F(x)=g(x)-h(x),
\]
where
\[
g(x)
=
\sum_{i=1}^m
\sum_{r=1}^k \rho_F(x^r-a^i),
\qquad
h(x)
=
\sum_{i=1}^m
\max_{\ell=1,\ldots,k}
\sum_{\substack{r=1 \\ r\neq \ell}}^k
\rho_F(x^r-a^i).
\]

This DC structure enables the use of the DC Algorithm (DCA),
introduced in~\cite{TaoAn97}, which has been widely applied to nonsmooth and
nonconvex optimization problems.
Enhanced variants, including boosted and adaptive schemes, have been proposed;
see, e.g., \cite{AragonArtachoFlemingVuong18, AragonArtachoVuong20, LongNamTranVan24}.
In particular, smoothing-based DC methods were developed in
\cite{LongNamTranVan24} to handle the nonsmooth structure of the objective
function. Recently, the case where the number of centers is unknown has been investigated 
in clustering models based on the minimum sum-of-squares criterion; see, e.g., \cite{An25}. 
In the context of multifacility location problems, related regularized formulations 
within the generalized Weber framework have been studied in 
\cite{GeremewLongNamSolanoHerrera26}. 

Despite these developments, DC-based approaches typically require solving
convex subproblems at each iteration, which can be computationally demanding
for large-scale instances. This motivates the search for alternative approaches
that avoid solving subproblems at each iteration.

From an optimization viewpoint, proximal-type techniques play a central role
in the treatment of nonsmooth problems.
Among them, the Moreau envelope provides a classical and effective smoothing
mechanism for proper lower semicontinuous functions under suitable regularity
assumptions; see, e.g.,~\cite{Moreau65, 
PoliquinRockafellar96,RockafellarWets98}.
In the convex setting, the Moreau envelope is continuously differentiable with
a Lipschitz continuous gradient and preserves minimizers.
In the nonconvex setting, local smoothness properties can still be guaranteed
under assumptions such as prox-regularity and calmness
\cite{PoliquinRockafellar96, RockafellarWets98}.

Motivated by these considerations, we propose a first-order algorithm for
computing stationary points of problem~\eqref{maxminn} based on Moreau
envelope smoothing.
By applying the Moreau envelope to the nonsmooth terms in the objective function,
we obtain a continuously differentiable approximation that enables the use of
simple gradient-based iterations.
The proposed method is easy to implement and well suited for large-scale
problems.
Moreover, starting from an initial overestimation of the number of
facilities, the algorithm is able to adaptively adjust the number of clusters
in practice during the optimization process.

The remainder of the paper is organized as follows.
Section~\ref{sec2} presents preliminary results and basic tools from
variational analysis.
Section~\ref{sec3} introduces the proposed algorithm and establishes its
convergence properties.
Numerical experiments illustrating the effectiveness of the proposed method
are reported in Section~\ref{sec4}.

\section{Preliminaries and Technical Tools}\label{sec2}

In this section, we recall several notions from variational analysis that will be used to establish local smoothness properties of nonsmooth functions arising in our model.

Throughout this paper, we work in the Euclidean space $\mathbb{R}^n$, equipped with the inner product $\langle \cdot, \cdot \rangle$ and the Euclidean norm $\|\cdot\|$. As usual, $\|\cdot\|_1$ and $\|\cdot\|_\infty$ denote the $\ell_1$-norm and the $\ell_\infty$-norm in $\mathbb{R}^n$, respectively.

Let $\varphi:\mathbb{R}^n \to (-\infty,+\infty]$ be a proper and lower semicontinuous
function. For a given parameter $\mu>0$, the \emph{Moreau envelope} of $\varphi$
is defined (see \cite{Moreau65,RockafellarWets98}) by
\[
\varphi^\mu(x)
:=
\inf_{y\in\mathbb{R}^n}
\left\{
\varphi(y)
+
\frac{1}{2\mu}\|y-x\|^2
\right\},
\qquad x\in\mathbb{R}^n.
\]
The associated \emph{proximal mapping} is given by
\[
\mathrm{prox}_{\mu\varphi}(x)
:=
\arg\min_{y\in\mathbb{R}^n}
\left\{
\varphi(y)
+
\frac{1}{2\mu}\|y-x\|^2
\right\}.
\]

It is well known that if $\varphi$ is convex, then $\varphi^\mu$ is continuously 
differentiable on $\mathbb{R}^n$, and
\[
\nabla \varphi^\mu(x)
=
\frac{1}{\mu}\bigl(x-\mathrm{prox}_{\mu\varphi}(x)\bigr),
\]
see, e.g., \cite[Theorem 2.26]{RockafellarWets98}.

For nonconvex functions, differentiability may fail in general; however,
under suitable local regularity conditions, such as local single-valuedness
and continuity of the proximal mapping, the Moreau envelope remains locally smooth.

To formulate precise regularity conditions ensuring local differentiability,
we recall the notions of prox-boundedness and prox-regularity.

\begin{Definition}[{\cite[Definition~1.23]{RockafellarWets98}}]\label{prox_bounded}
A function $\varphi:\mathbb{R}^n\to(-\infty,+\infty]$ is said to be
\emph{prox-bounded} if there exists $\mu>0$ such that
\[
\inf_{x\in\mathbb{R}^n}
\left\{
\varphi(x)+\frac{1}{2\mu}\|x\|^2
\right\}
>
-\infty.
\]
The supremum of all such $\mu$ is called the \emph{threshold of prox-boundedness}
of $\varphi$.
\end{Definition}

For any $\mu$ below this threshold, the Moreau envelope is well-defined.

\begin{Definition}[{\cite[Definition~13.27]{RockafellarWets98}}]\label{prox_regular}
Let $\varphi:\mathbb{R}^n\to(-\infty,+\infty]$ and let
$\bar x\in\mathbb{R}^n$ with $\bar v\in\partial\varphi(\bar x)$.
The function $\varphi$ is said to be \emph{prox-regular at $\bar x$ for $\bar v$}
if $\varphi$ is finite and locally lower semicontinuous at $\bar x$,
$\bar v\in\partial\varphi(\bar x)$, and there exist constants
$\varepsilon>0$ and $\rho\ge0$ such that
\[
\varphi(x')
\ge
\varphi(x) + \langle v, x'-x\rangle - \frac{\rho}{2}\|x'-x\|^2
\]
for all $x',x\in\mathbb{B}(\bar x,\varepsilon)$ and all
$v\in\partial\varphi(x)$ satisfying
$\|v-\bar v\|<\varepsilon$ and
$\varphi(x)<\varphi(\bar x)+\varepsilon$.
Moreover, if this property holds for all $\bar v\in\partial\varphi(\bar x)$,
then $\varphi$ is said to be prox-regular at $\bar x$.
\end{Definition}

The following lemma provides a local differentiability property of the Moreau envelope. This is a key technical tool in our convergence analysis. For completeness, we introduce the proof here. 

\begin{Lemma} \label{lem:local-moreau-diff}
Let $\varphi:\mathbb{R}^n \to (-\infty,+\infty]$ be proper and lower semicontinuous,
and let $\mu>0$. Suppose that the proximal mapping $\mathrm{prox}_{\mu\varphi}$
is single-valued and continuous in a neighborhood $U$ of $\bar x \in \mathbb{R}^n$.
Then the Moreau envelope $\varphi^\mu$ is continuously differentiable on $U$, and
\[
\nabla \varphi^\mu(x)
=
\frac{1}{\mu}\bigl(x-\mathrm{prox}_{\mu\varphi}(x)\bigr)
\quad \text{for all } x \in U.
\]
\end{Lemma}

\begin{proof}
Let $U$ be a neighborhood of $\bar x$ on which $\mathrm{prox}_{\mu\varphi}$
is single-valued and continuous, and define
\[
y(x) := \mathrm{prox}_{\mu\varphi}(x), \quad x \in U.
\]

By definition of the Moreau envelope,

\[
\varphi^\mu(x)
=
\inf_{y\in\mathbb{R}^n}
\left\{
\varphi(y)
+
\frac{1}{2\mu}\|y-x\|^2
\right\},
\]

and the minimum is attained at $y(x)$, so that
\[
\varphi^\mu(x)
=
\varphi(y(x)) + \frac{1}{2\mu}\|y(x)-x\|^2.
\]

For any $h \in \mathbb{R}^n$, we have

\begin{align}
\varphi^\mu(x+h)
&\le
\varphi(y(x)) + \frac{1}{2\mu}\|y(x)-x-h\|^2 \notag\\
&=
\varphi(y(x)) + \frac{1}{2\mu}\Big(\|y(x)-x\|^2 - 2\langle y(x)-x, h\rangle + \|h\|^2\Big) \notag\\
&=
\varphi(y(x)) + \frac{1}{2\mu}\|y(x)-x\|^2
+ \frac{1}{\mu}\langle x - y(x), h \rangle
+ \frac{1}{2\mu}\|h\|^2 \notag\\
&=
\varphi^\mu(x)
+ \frac{1}{\mu}\langle x - y(x), h \rangle
+ \frac{1}{2\mu}\|h\|^2
\label{eq1}
\end{align}

On the other hand, by definition of $\varphi^\mu(x)$,
\[
\varphi^\mu(x)
\le
\varphi(y(x+h)) + \frac{1}{2\mu}\|y(x+h)-x\|^2.
\]
Hence,
\begin{align*}
\varphi^\mu(x+h) - \varphi^\mu(x)
&\ge
\left[
\varphi(y(x+h)) + \frac{1}{2\mu}\|y(x+h)-(x+h)\|^2
\right] \\
&\quad -
\left[
\varphi(y(x+h)) + \frac{1}{2\mu}\|y(x+h)-x\|^2
\right] \\
&=
\frac{1}{2\mu}\Big(
\|y(x+h)-(x+h)\|^2 - \|y(x+h)-x\|^2
\Big)\\
&= \frac{1}{2\mu}\Big( 2\langle x - y(x+h), h\rangle + \|h\|^2 \Big),
\end{align*}
where the last equation is obtained using the following identity
\[
\|a-h\|^2 - \|a\|^2
=
-2\langle a,h\rangle + \|h\|^2,
\quad \text{with } a = y(x+h)-x.
\]
Therefore,
\begin{align}
\varphi^\mu(x+h) - \varphi^\mu(x)
&\ge
\frac{1}{\mu}\langle x - y(x+h), h \rangle
+ \frac{1}{2\mu}\|h\|^2. \label{eq2}
\end{align}
Combining \eqref{eq1} and \eqref{eq2}, we obtain
\begin{align*}
\varphi^\mu(x)
+ \left\langle \frac{1}{\mu}(x-y(x+h)), h \right\rangle
- \frac{1}{2\mu}\|h\|^2
\le \varphi^\mu(x+h) 
\le
\varphi^\mu(x)
+ \left\langle \frac{1}{\mu}(x-y(x)), h \right\rangle
+ \frac{1}{2\mu}\|h\|^2.
\end{align*}

Observe that 
\[
\langle x-y(x+h),h\rangle
=
\langle x-y(x),h\rangle
+
\langle y(x)-y(x+h),h\rangle.
\]
Since $y$ is continuous on $U$, we have $y(x+h)\to y(x)$ as $h\to 0$, hence
\[
\|y(x+h)-y(x)\| \to 0.
\]
Therefore,
\[
|\langle y(x+h)-y(x), h \rangle|
\le \|y(x+h)-y(x)\|\,\|h\|.
\]
Dividing both sides by $\|h\|$ (for $h\neq 0$), we obtain
\[
\frac{|\langle y(x+h)-y(x), h \rangle|}{\|h\|}
\le \|y(x+h)-y(x)\| \to 0,
\]
which shows that
\[
\langle y(x+h)-y(x), h \rangle = o(\|h\|).
\]
Substituting this into the lower bound and using the continuity of $y$, we obtain
\[
\varphi^\mu(x+h)
\ge
\varphi^\mu(x)
+ \frac{1}{\mu}\langle x-y(x), h \rangle
+ o(\|h\|).
\]
Together with the upper bound, this yields
\[
\varphi^\mu(x+h)
=
\varphi^\mu(x)
+ \frac{1}{\mu}\langle x-y(x), h \rangle
+ o(\|h\|).
\]
Hence, $\varphi^\mu$ is Fr\'echet differentiable at $x$, with
\[
\nabla \varphi^\mu(x)
=
\frac{1}{\mu}(x-y(x)).
\]
The continuity of $\nabla \varphi^\mu$ follows from the continuity of $y(x)$.
\end{proof}
\section{Proposed Algorithm and Convergence Analysis}\label{sec3}

In this section, we first reformulate the problem using the Minkowski functional 
and analyze its structural properties. Based on these results, we then construct 
a smoothing approach via the Moreau envelope and propose an adaptive algorithm, 
for which convergence properties are established.

\begin{Lemma}[Lemma~2.5, \cite{LongNamTranVan24}]\label{lemmarho0}
Let $F$ be a nonempty compact convex set of $\mathbb{R}^n$, $\rho_F$ be the Minkowski functional associated with $F$, and $F^\circ$ its polar set, defined by
\[
F^{\circ} = \{v \in \mathbb{R}^n \mid \langle v, x\rangle \le 1 \;\text{for all } x \in F\}.
\]
Then the following properties hold:
\begin{enumerate}
\item[{\rm (a)}] $\rho_F(\alpha x) = \alpha \rho_F(x)$ for all $\alpha \ge 0$,
$x \in \mathbb{R}^n$.

\item[{\rm (b)}] $\rho_F(x_1 + x_2)
\le \rho_F(x_1) + \rho_F(x_2)$
for all $x_1, x_2 \in \mathbb{R}^n$.

\item[{\rm (c)}] $\rho_F(x) = 0$ if and only if $x = 0$.

\item[{\rm (d)}] $\rho_F$ is Lipschitz continuous on $\mathbb{R}^n$ with Lipschitz constant $\|F^\circ\|$.

\item[{\rm (e)}] $x \in {\rm bd}(F)$ if and only if $\rho_F(x) = 1$,
where ${\rm bd}(F)$ denotes the boundary of $F$.

\item[{\rm (f)}] $\rho_{F^\circ}(v)
= \max_{u \in F} \langle v, u \rangle$
for all $v \in \mathbb{R}^n$.

\item[{\rm (g)}] 
\[
\frac{\rho_F(x)}{\|F^\circ\|}
\le \|x\|
\le \|F\| \rho_F(x),
\qquad \forall x \in \mathbb{R}^n.
\]

\item[{\rm (h)}] 
$\rho_F(x)
\le \|F\| \|F^\circ\| \rho_F(-x)$
for all $x \in \mathbb{R}^n$.
\end{enumerate}
\end{Lemma}

To reformulate problem~\eqref{maxminn}, for each $i=1,\ldots,m$, we introduce
\[
\varphi_i(x)
:=
\min_{\ell=1,\ldots,k} \rho_F(x^\ell-a^i),
\qquad
x=(x^1,\ldots,x^k)\in\mathbb{R}^{nk}.
\]
With this notation, problem~\eqref{maxminn} is equivalent to
\begin{equation}\label{sum_phi}
\min_{x\in\mathbb{R}^{nk}} f_F(x)
:=
\sum_{i=1}^m \varphi_i(x).
\end{equation}
Let $S$ denote the set of all global minimizers of~\eqref{sum_phi}. Following~\cite{LongNamTranVan24}, the solution set $S$ is nonempty and compact. To analyze the local structure of the problem, in particular around potential local minimizers,
we associate with each configuration $x=(x^1,\ldots,x^k)\in\mathbb{R}^{nk}$
a family of subsets $A^1,\ldots,A^k$ of the demand set $A$, constructed sequentially as follows.

Let $A^0=\emptyset$ and define
\[
A^\ell
=
\left\{
a^i \in A \setminus \Bigl(\bigcup_{q=0}^{\ell-1} A^q\Bigr)
\;\Big|\;
\rho_F(x^\ell-a^i)
=
\min_{r=1,\ldots,k} \rho_F(x^r-a^i)
\right\},
\qquad
\ell=1,\ldots,k.
\]
The collection $\{A^1,\ldots,A^k\}$ is referred to as the
\emph{natural clustering} induced by $x$.

\begin{Definition}[\cite{LongNamTranVan24}, Definition~3.5]
A component $x^\ell$ of $x=(x^1,\ldots,x^k)\in\mathbb{R}^{nk}$ is called
\emph{attractive} with respect to $A$ if the associated set
\begin{equation}\label{attraction}
A[x^\ell]
=
\left\{
a^i\in A
\;\Big|\;
\rho_F(x^\ell-a^i)
=
\min_{r=1,\ldots,k} \rho_F(x^r-a^i)
\right\}
\end{equation}
is nonempty.
\end{Definition}

By construction, the subsets $A^\ell$ satisfy
\[
A^\ell
=
A[x^\ell]
\setminus
\Bigl(\bigcup_{q=1}^{\ell-1} A^q\Bigr),
\qquad \ell=1,\ldots,k,
\]
which ensures that each demand point is assigned to exactly one cluster.

\begin{Lemma}[\cite{LongNamTranVan24}, Lemma~3.8]\label{lem:stability}
Let $\bar{x}=(\bar{x}^1,\ldots,\bar{x}^k)\in\mathbb{R}^{nk}$.
Assume that, for every $i=1,\ldots,m$, the index set
\[
L_i(x)
:=
\left\{
\ell\in\{1,\ldots,k\}
\;\big|\;
a^i\in A[x^\ell]
\right\}
\]
is a singleton at $\bar{x}$.
Then there exists $\varepsilon>0$ such that, for any
$x=(x^1,\ldots,x^k)\in\mathbb{R}^{nk}$ satisfying
$\|x^r-\bar{x}^r\|<\varepsilon$ for all $r=1,\ldots,k$, the following properties hold:
\begin{enumerate}
\item[(i)] $L_i(x)=L_i(\bar{x})$ for all $i=1,\ldots,m$;
\item[(ii)] $A[x^r]=A[\bar{x}^r]$ for all $r=1,\ldots,k$.
\end{enumerate}
\end{Lemma}

The objective function $f_F$ is nonsmooth and nonconvex due to the presence of the minimum operator.
To overcome this difficulty, and motivated by the local structure established above,
we adopt a smoothing strategy based on the \emph{Moreau envelope} introduced in Section~\ref{sec2}.

Given $\mu>0$, the Moreau envelope of $\varphi_i$ is defined by
\[
\varphi_i^\mu(x)
:=
\inf_{y\in\mathbb{R}^{nk}}
\left\{
\varphi_i(y)
+
\frac{1}{2\mu}\|y-x\|^2
\right\}.
\]

\begin{Lemma}\label{lem:prox-regular}
Fix $i\in\{1,\ldots,m\}$ and define $\varphi:=\varphi_i$.
Let $\bar{x}\in\mathbb{R}^{nk}$.
Assume that the index set $L_i(\bar{x})$ is a singleton, i.e., the minimum is attained uniquely.
Then $\varphi$ is prox-regular at $\bar{x}$ for $0$ and prox-bounded.
Consequently, there exists $\bar\mu>0$ such that, for all
$\mu\in(0,\bar\mu)$, the proximal mapping $\mathrm{prox}_{\mu\varphi}$
is single valued and Lipschitz continuous in a neighborhood of $\bar{x}$.
\end{Lemma}

\begin{proof} Fix $i\in\{1,\ldots,m\}$ and assume that
\[ L_i(\bar{x})=\{\ell^\ast\}. \] 
Recall that 
\[ \varphi(x) = \min_{\ell=1,\ldots,k} \rho_F(x^\ell-a^i), \qquad x=(x^1,\ldots,x^k)\in\mathbb{R}^{nk}. \] 
By Lemma~\ref{lem:stability}, there exists a neighborhood $U$ of $\bar{x}$ such that, for every $x\in U$, the minimum defining $\varphi(x)$ is attained uniquely at the same index $\ell^\ast$. Hence, \begin{equation}\label{eq:local_representation} \varphi(x)=\rho_F(x^{\ell^\ast}-a^i), \qquad \forall x\in U. \end{equation} 

According to Lemma~\ref{lemmarho0}{\rm(a)}--{\rm(d)}, the Minkowski functional $\rho_F$ is finite-valued, lower semicontinuous, convex, and globally Lipschitz continuous on $\mathbb{R}^n$. 

Define the function $g:\mathbb{R}^{nk}\to\mathbb{R}$ by 
\[ g(x):=\rho_F(x^{\ell^\ast}-a^i). \]
Let $T:\mathbb{R}^{nk}\to\mathbb{R}^n$ be the affine mapping \[ T(x)=x^{\ell^\ast}-a^i. \] 
Since $\rho_F$ is proper, lower semicontinuous, and convex, it follows that \[ g=\rho_F\circ T \] is a proper, lower semicontinuous, and convex function on $\mathbb{R}^{nk}$. 

By \eqref{eq:local_representation}, the function $\varphi$ coincides with the convex function $g$ in a neighborhood of $\bar{x}$. Therefore, $\varphi$ is locally convex around $\bar{x}$. In particular, by \cite[Example~13.30 and Proposition~13.32]{RockafellarWets98}, $\varphi$ is prox-regular at $\bar{x}$ for $0$. 

Moreover, since $\rho_F\ge 0$, the function $\varphi$ is finite everywhere and bounded from below. Hence, $\varphi$ is prox-bounded. 

Finally, by \cite[Theorem~13.37]{RockafellarWets98}, there exists $\bar\mu>0$ such that, for all $\mu\in(0,\bar\mu)$, the proximal mapping $\mathrm{prox}_{\mu\varphi}$ is single valued and Lipschitz continuous in a neighborhood of $\bar{x}$. 
\end{proof} 

The above analysis shows that, under the uniqueness condition on the index sets,
the objective function $f_F$ admits a locally stable structure around $\bar x$,
in the sense that the active indices defining each $\varphi_i$ remain unchanged.
Consequently, in a neighborhood of $\bar x$, the nonsmooth function $f_F$
reduces to a structured finite sum with fixed active indices.

By Lemma~\ref{lem:local-moreau-diff}, for each $i=1,\ldots,m$, the Moreau envelope $\varphi_i^\mu$ is continuously differentiable in a neighborhood of $\bar x$, with gradient given by 
\[ \nabla \varphi_i^\mu(x) = \frac{1}{\mu}\bigl(x-\mathrm{prox}_{\mu\varphi_i}(x)\bigr). \]
Since $f_F^\mu=\sum_{i=1}^m \varphi_i^\mu$ is a finite sum of continuously differentiable functions, it follows that $f_F^\mu$ is continuously differentiable in a neighborhood of $\bar x$, with gradient 
\[ \nabla f_F^\mu(x) = \sum_{i=1}^m \nabla \varphi_i^\mu(x) = \frac{1}{\mu} \sum_{i=1}^m \bigl(x-\mathrm{prox}_{\mu\varphi_i}(x)\bigr). \] 
The explicit gradient formula makes it possible to apply standard gradient-based methods to the smoothed problem. 

We now propose a general Moreau-envelope-based algorithm for solving problem~\eqref{sum_phi}. The method applies a gradient step on the Moreau-smoothed objective, followed by reassignment of demand points and deletion of any empty clusters (i.e., centers that do not serve any demand point). This approach allows the number of clusters to decrease dynamically whenever a cluster becomes inactive

\begin{algorithm}[H]
\caption{General Moreau-Envelop Algorithm for GMWP}
\label{alg:adaptive-moreau-gmw}

\textbf{Input:} 
$A := \{a^1,\ldots,a^m\} \subset \mathbb{R}^n$, 
initial number of clusters $k_0$, 
smoothing parameter $\mu>0$, 
stepsize $\alpha>0$, and 
tolerance $\varepsilon>0$.

\textbf{Initialization:} 
Choose initial centers $x^{(0)}\in\mathbb{R}^{nk_0}$, 
set $k^{(0)} := k_0$, 
$t := 0$.

\begin{algorithmic}[1]
\Repeat
    \State Compute $\nabla f_{F,k_t}^{\mu}(x^{(t)})$. 
    \State Gradient update: $x^{(t+\frac12)} = x^{(t)} - \alpha \nabla f_{F,k_t}^{\mu}(x^{(t)})$.
    \State Reassign demand points.
    \State Delete empty clusters.
    \State Update the current number of clusters $k_t$ .
    \State $x^{(t+1)} \leftarrow x^{(t+\frac12)}$.
    \State $t \leftarrow t + 1$.
\Until{$\|\nabla f_{F,k_t}^{\mu}(x^{(t)})\| \le \varepsilon (\|x^{(t)}\| + 1)$.}
\end{algorithmic}
\end{algorithm}


\begin{theorem}\label{thm:descent-stabilization}
Assume that there exists an iteration index $t_0 \in \mathbb{N}$ such that the assignment structure remains unchanged for all $t \ge t_0$. If the stepsize satisfies $\alpha \in (0, 2\mu/m)$, then:
\begin{enumerate}
    \item[(i)] The sequence $\{f_{F,k_t}^{\mu}(x^{(t)})\}$ is nonincreasing;
    \item[(ii)] The number of cluster eliminations is finite;
    \item[(iii)] $\lim_{t \to \infty} \|\nabla f_{F,\bar k}^{\mu}(x^{(t)})\| = 0$, 
    where $\bar k$ is the stabilized number of clusters.
\end{enumerate}
\end{theorem}

\begin{proof}
(i) Between two consecutive elimination events, the number of clusters is fixed, say $k_t$. On such intervals, the function $f_{F,k_t}^{\mu}$ is continuously differentiable and its gradient is Lipschitz continuous with constant
\[
L = \frac{m}{\mu}.
\] 
Applying the standard descent lemma to the gradient update
\[
x^{(t+1)} = x^{(t)} - \alpha \nabla f_{F,k_t}^{\mu}(x^{(t)}),
\] 
we obtain
\[
f_{F,k_t}^{\mu}(x^{(t+1)}) 
\le f_{F,k_t}^{\mu}(x^{(t)}) 
- \alpha \Bigl(1 - \frac{\alpha L}{2}\Bigr) \|\nabla f_{F,k_t}^{\mu}(x^{(t)})\|^2.
\]
Since $\alpha \in (0, 2/L)$, the right-hand side is nonincreasing. Elimination steps remove only empty clusters and do not increase the objective, so the overall sequence $\{f_{F,k_t}^{\mu}(x^{(t)})\}$ is nonincreasing.

(ii) The integer sequence $\{k_t\}$ is bounded below by 1 and strictly decreases whenever a cluster is eliminated. Therefore, only finitely many elimination events can occur. Consequently, there exists $\bar k \in \mathbb{N}$ and an iteration index $t_1$ such that $k_t = \bar k$ for all $t \ge t_1$.

(iii) By assumption, the assignment structure is fixed for all $t \ge t_0$, and by (ii) the number of clusters stabilizes at $\bar k$. For all sufficiently large $t$, the algorithm reduces to gradient descent applied to the smooth function $f_{F,\bar k}^{\mu}$ with stepsize $\alpha \in (0,2/L)$. Standard convergence results for gradient descent yield
\[
\lim_{t \to \infty} \|\nabla f_{F,\bar k}^{\mu}(x^{(t)})\| = 0.
\]
\end{proof}

\begin{Corollary}\label{cor:stationary-point}
Under the assumptions of Theorem~\ref{thm:descent-stabilization}, every accumulation point $\bar x$ of the sequence $\{x^{(t)}\}$ is a stationary point of the smoothed problem with $\bar k$ clusters, that is,
\[
\nabla f_{F,\bar k}^{\mu}(\bar x) = 0.
\]
\end{Corollary}

\begin{proof}
Let $\bar x$ be any accumulation point of $\{x^{(t)}\}$. 
By Theorem~\ref{thm:descent-stabilization}(iii), $\|\nabla f_{F,\bar k}^{\mu}(x^{(t)})\| \to 0$, 
and by continuity of $\nabla f_{F,\bar k}^{\mu}$, we have 
\[
\nabla f_{F,\bar k}^{\mu}(\bar x) = 0.
\]
\end{proof}

\section{Numerical Results}
\label{sec4}

This section reports numerical experiments evaluating the performance of
the proposed Moreau-envelope algorithm for the generalized multi-source Weber problem.
The experiments are designed to address two main questions:
(i) how effectively the Moreau envelope-based smoothing strategy handles the
nonsmooth and nonconvex structure of the objective, and
(ii) whether the proposed adaptive mechanism can automatically reduce the
number of clusters while maintaining competitive objective values.

We compare the proposed method with the standard DCA and the adaptive BDCA
introduced in~\cite{LongNamTranVan24} under three norms: $\ell_1$, $\ell_2$, and $\ell_\infty$.

To facilitate efficient implementation, we exploit the structure of the component functions:
\[
\varphi_i(x) = \min_{\ell=1,\dots,k} \rho_F(x^\ell-a^i),
\]
where $\ell_i^\ast \in \arg\min_{\ell=1,\dots,k} \rho_F(x^\ell-a^i)$ denotes the closest center. 
Under the stability condition of Lemma~\ref{lem:stability}, this index is locally unique, so the proximal mapping $\mathrm{prox}_{\mu\varphi_i}(x)$ modifies only the active block $x^{\ell_i^\ast}$. 
Consequently, the gradient of the Moreau envelope is given by
\[
\nabla \varphi_i^\mu(x)
=
\frac{1}{\mu} \bigl(x-\mathrm{prox}_{\mu\varphi_i}(x)\bigr),
\]
with nonzero entries only in the active block, which significantly reduces computational cost.

For each norm $\rho_F$, the active-block proximal mapping admits the following forms:

\begin{itemize}
    \item $\ell_1$: $\mathrm{prox}_{\mu\|\cdot-a^i\|_1}(x^{\ell_i^\ast}) = a^i + \mathcal{S}_\mu(x^{\ell_i^\ast}-a^i)$, 
    where $\mathcal{S}_\mu$ is the componentwise soft-thresholding operator.
    
    \item $\ell_2$: $\mathrm{prox}_{\mu\|\cdot-a^i\|_2}(x^{\ell_i^\ast}) = a^i + \left(1-\frac{\mu}{\|x^{\ell_i^\ast}-a^i\|_2}\right)_+(x^{\ell_i^\ast}-a^i)$, the classical Euclidean shrinkage operator.
    
    \item $\ell_\infty$: $\mathrm{prox}_{\mu\|\cdot-a^i\|_\infty}(x^{\ell_i^\ast}) = a^i + \Pi_{\|\cdot\|_1\le\mu}(x^{\ell_i^\ast}-a^i)$, the projection onto the $\ell_1$ ball of radius $\mu$.
\end{itemize}

Two experimental settings are considered:

\begin{itemize}
    \item Fixed number of centers: The initial number of centers $k$ is specified. 
    Centers that do not serve any demand point are removed, allowing the number of centers to decrease over iterations.

    \item Adaptive number of centers: The algorithm includes a cluster-merging mechanism based on the objective function. 
    The number of centers may decrease due to either empty clusters or merging operations.
\end{itemize}

In all cases, the Moreau envelope $f_F^\mu$ is continuously differentiable with Lipschitz continuous gradient of constant $L=m/\mu$. 
This allows the same stepsize strategy and stopping criterion to be used across all norms. 
All algorithms are run until either the relative gradient norm satisfies
\[
\|\nabla f(x)\| \le \epsilon (\|x\|+1), \qquad \epsilon = 10^{-6},
\]
or a maximum of $10^4$ iterations is reached.

All experiments are implemented in \textsc{Matlab}~R2025b (Windows~11, 16-core AMD Ryzen processor, 16~GB of RAM). To ensure reproducibility, the random seed is fixed to \texttt{2026} for all experiments. Two types of datasets are considered:
\begin{itemize}
    \item Synthetic dataset: Generated in $\mathbb{R}^2$ as a mixture of $k_{\mathrm{true}} = 6$ Gaussian clusters, each containing $50$ points. Gaussian noise with zero mean and covariance $0.3^2 I_2$ is added, resulting in a total of $m = 300$ points.
    
\item Real datasets: Three benchmark datasets are used, namely 
Iris $(m = 150, d = 4, k_{\mathrm{true}} = 3)$, 
Wine $(m = 178, d = 13, k_{\mathrm{true}} = 3)$, 
and Glass $(m = 214, d = 9, k_{\mathrm{true}} = 6)$,
where $m$, $d$, and $k_{\mathrm{true}}$ denote the number of data points, 
the number of features, and the true number of classes, respectively.
\end{itemize}

For each example, the algorithms are run 50 times with different initial points. 
The reported results include the best solution over all runs for visualization, 
while averages over runs are used for timing and other metrics.
This setup ensures fair comparisons between the Moreau, DCA, and adaptive BDCA methods across different norms and initializations.

In the implementation, we employ a continuation strategy on the smoothing parameter $\mu$. 
Smaller values of $\mu$ yield a closer approximation to the original nonsmooth problem. 
The algorithm is applied over the decreasing sequence
\[
\mu \in \{1,\,0.5,\,0.2,\,0.1,\,0.01\},
\]
where the solution at each stage serves as the initialization for the next. 
The stepsize at each stage is chosen proportionally to $\mu$:
\[
\alpha = \frac{1.8\,\mu}{m}.
\]
The maximum number of iterations for all methods is set to $10{,}000$. 
For the Moreau-envelope algorithm, the total number of iterations is the sum over all stages. 
For DCA and BDCA, the total number of iterations is calculated by summing all inner iterations. The parameters for the standard DCA and adaptive BDCA are taken from~\cite{LongNamTranVan24}. 
Specifically, we use $\alpha = 0.05$, $\beta = 0.01$, $\gamma = 2$, and $\mu = 1$. 

\subsection{Fixed Number of Centers} \label{fixed-synthetic-k}

In this setting, the number of initial centers $k$ is specified and remains fixed throughout the optimization, with any center not assigned to a demand point being eliminated. The fixed-$k$ implementation of the proposed method is summarized in Algorithm~\ref{alg:fixed-moreau-gmw-1}.

We evaluate the algorithm for two fixed-$k$ settings, $k_{\mathrm{init}} = 10$ and $k_{\mathrm{init}} = 20$, performing 50 runs with randomly initialized centers for each setting. For each method, the best solution over all runs is selected based on the objective value.

For these best solutions, we report the clustering accuracy (ACC), objective value, computational time, and the number of clusters obtained after eliminating empty centers. Computational time is measured in seconds. ACC measures the percentage of data points correctly assigned to their corresponding ground-truth clusters, providing an external evaluation of clustering quality.

\begin{algorithm}[H]
\caption{Moreau-Envelop Algorithm for GMWP with Fixed Number of Centers}
\label{alg:fixed-moreau-gmw-1}

\textbf{Input:}
Demand points $A := \{a^1,\ldots,a^m\} \subset \mathbb{R}^n$,
fixed number of centers $k$,
smoothing parameter $\mu > 0$,
stepsize $\alpha > 0$, and tolerance $\varepsilon > 0$.

\textbf{Initialization:}
Choose an initial point $x^{(0)} \in \mathbb{R}^{nk}$, 
set $k_0 := k$, 
$t := 0$.

\begin{algorithmic}[1]
\Repeat
    \State Compute the gradient of the Moreau envelope:
    \[
    \nabla f_{F,k_t}^{\mu}(x^{(t)})
    =
    \frac{1}{\mu} \sum_{i=1}^m
    \Bigl(x^{(t)} - \mathrm{prox}_{\mu \varphi_i}(x^{(t)}) \Bigr),
    \]
    where $\mathrm{prox}_{\mu \varphi_i}(x^{(t)})$ modifies only the block
    corresponding to the center nearest to $a^i$.

    \State Gradient update:
    \[
    x^{(t+\frac12)} = x^{(t)} - \alpha \nabla f_{F,k_t}^{\mu}(x^{(t)}).
    \]

    \State Reassign demand points to the nearest centers.

    \State Delete empty clusters if any exist.

    \State Update the current number of centers $k_t$.

    \State $x^{(t+1)} \leftarrow x^{(t+\frac12)}$.

    \State $t \leftarrow t + 1$.
\Until{
\[
\|\nabla f_{F,k_t}^{\mu}(x^{(t)})\| \le \varepsilon (\|x^{(t)}\| + 1).
\]
}
\end{algorithmic}
\end{algorithm}

The numerical results are summarized in Tables~\ref{tab:fixed_k_10} and \ref{tab:fixed_k_20} for $k_{\mathrm{init}} = 10$ and $k_{\mathrm{init}} = 20$, respectively. Corresponding visualizations of the best solution over 50 runs on the synthetic dataset are shown in Figures~\ref{fig:fixed_k_10} and \ref{fig:fixed_k_20}.

\begin{table}[H]
\centering
\setlength{\tabcolsep}{3pt} 

\renewcommand{\arraystretch}{1}
\begin{tabular}{c|l|cccc|cccc|cccc}
\hline
& & \multicolumn{12}{c}{\textbf{Norm}} \\
\textbf{Dataset} & \textbf{Method}
& \multicolumn{4}{c}{$\ell_1$}
& \multicolumn{4}{c}{$\ell_2$}
& \multicolumn{4}{c}{$\ell_\infty$} \\
\cline{3-14}
& & ACC & Obj & Time & $k$ 
& ACC & Obj & Time & $k$ 
& ACC & Obj & Time & $k$ \\
\hline

\multirow{3}{*}{Synthetic}
& Moreau
& \textbf{0.9800} & \textbf{120} & \textbf{0.01} & 10
& \textbf{0.9700} & 120 & \textbf{0.01} & 10
& \textbf{0.9767} & \textbf{144} & \textbf{21} & 10 \\

& DCA
& 0.8733 & 183 & 20 & 7
& 0.9667 & 104 & 29 & 10
& 0.5033 & 205 & 103 & 4 \\

& aBDCA
& 0.8733 & 184 & 4 & 7
& 0.9667 & 104 & 4 & 10
& 0.5033 & 205 & 117 & 4 \\

\hline

\multirow{3}{*}{Iris}
& Moreau
& \textbf{0.9600} & \textbf{126} & \textbf{0.01} & 10
& 0.9467 & \textbf{77} & \textbf{0.01} & 10
& \textbf{0.9133} & \textbf{87} & \textbf{10} & 8 \\

& DCA
& 0.8733 & 203 & 3 & 5
& \textbf{0.9533} & 78 & 2 & 10
& 0.8267 & 102 & 30 & 4 \\

& aBDCA
& 0.8733 & 203 & 1 & 7
& 0.9467 & 78 & 2 & 10
& 0.8267 & 102 & 38 & 4 \\
\hline

\multirow{3}{*}{Wine}
& Moreau
& 0.9551 & \textbf{1064} & \textbf{0.01} & 10
& \textbf{0.9775} & 396 & \textbf{0.01} & 10 
& \textbf{0.7472} & 293 & \textbf{10} & 10 \\

& DCA
& 0.9157 & 1296 & 5 & 6
& 0.9607 & 372 & 6 & 10 
& 0.6292 & \textbf{291} & 36 & 2 \\

& aBDCA
& \textbf{0.9719} & 1300 & 2 & 6
& 0.9607 & 372 & 9 & 10 
& 0.6292 & \textbf{291} & 49 & 2 \\
\hline

\multirow{3}{*}{Glass}
& Moreau
& \textbf{0.5701} & \textbf{636} & \textbf{0.02} & 10
& \textbf{0.6308} & 294 & \textbf{0.01} & 10
& \textbf{0.5234} & \textbf{281} & \textbf{12} & 10 \\

& DCA
& 0.5327 & 995 & 5 & 9
& 0.5374 & 274 & 8 & 10
& 0.4252 & 307 & 38 & 2 \\

& aBDCA
& 0.5000 & 1011 & 1 & 9
& 0.5374 & 274 & 8 & 10
& 0.4252 & 307 & 48 & 2 \\
\hline
\end{tabular}
\caption{Comparison of methods for the fixed-center setting with $k_{\mathrm{init}} = 10$.}
\label{tab:fixed_k_10}
\end{table}

\begin{table}[H]
\centering
\setlength{\tabcolsep}{3pt} 
\renewcommand{\arraystretch}{1}
\begin{tabular}{c|l|cccc|cccc|cccc}
\hline
& & \multicolumn{12}{c}{\textbf{Norm}} \\
\textbf{Dataset} & \textbf{Method}
& \multicolumn{4}{c}{$\ell_1$}
& \multicolumn{4}{c}{$\ell_2$}
& \multicolumn{4}{c}{$\ell_\infty$} \\
\cline{3-14}
& & ACC & Obj & Time & $k$ 
& ACC & Obj & Time & $k$ 
& ACC & Obj & Time & $k$ \\
\hline

\multirow{3}{*}{Synthetic}
& Moreau
& \textbf{0.9667} & \textbf{86} & \textbf{0.01} & 20
& \textbf{0.9633} & \textbf{89} & \textbf{0.01} & 20
& \textbf{0.9533} & \textbf{112} & \textbf{21} & 20 \\

& DCA
& 0.8633 & 181 & 20 & 9
& 0.9600 & 103 & 29 & 16
& 0.6567 & 103 & 212 & 4 \\

& aBDCA
& 0.8667 & 181 & 4 & 10
& 0.9533 & 104 & 4 & 15
& 0.6567 & 212 & 117 & 4 \\

\hline

\multirow{3}{*}{Iris}
& Moreau
& \textbf{0.9600} & \textbf{98} & \textbf{0.01} & 20
& 0.9400 & 63 & \textbf{0.01} & 20
& \textbf{0.9067} & \textbf{74} & \textbf{10} & 19 \\

& DCA
& 0.8600 & 200 & 6 & 9
& 0.9400 & 63 & 9 & 20
& 0.8267 & 102 & 52 & 4 \\

& aBDCA
& 0.8400 & 207 & 2 & 6
& \textbf{0.9667} & \textbf{62} & 6 & 20
& 0.8267 & 102 & 63 & 4 \\
\hline

\multirow{3}{*}{Wine}
& Moreau
& \textbf{0.9775} & \textbf{932} & \textbf{0.01} & 20
& \textbf{0.9607} & 357 & \textbf{0.01} & 20
& \textbf{0.8652} & \textbf{257} & \textbf{11} & 20 \\

& DCA
& 0.9663 & 1295 & 8 & 10
& 0.9270 & 318 & 13 & 20
& 0.6292 & 291 & 67 & 2 \\

& aBDCA
& 0.9719 & 1297 & 4 & 6
& 0.9270 & 318 & 21 & 20
& 0.6292 & 291 & 90 & 2 \\
\hline

\multirow{3}{*}{Glass}
& Moreau
& \textbf{0.6963} & \textbf{525} & \textbf{0.02} & 20
& 0.6682 & 246 & \textbf{0.02} & 20
& \textbf{0.6028} & \textbf{240} & \textbf{15} & 18 \\

& DCA
& 0.5140 & 989 & 11 & 10
& \textbf{0.6916} & \textbf{216} & 20 & 20
& 0.4252 & 307 & 110 & 2 \\

& aBDCA
& 0.5000 & 1002 & 3 & 11
& \textbf{0.6916} & \textbf{216} & 16 & 20
& 0.4252 & 307 & 101 & 2 \\
\hline

\end{tabular}
\caption{Comparison of methods for the fixed-center setting with $k_{\mathrm{init}} = 20$.}
\label{tab:fixed_k_20}
\end{table}

\begin{figure}[H]
\centering
\includegraphics[width=0.32\textwidth]{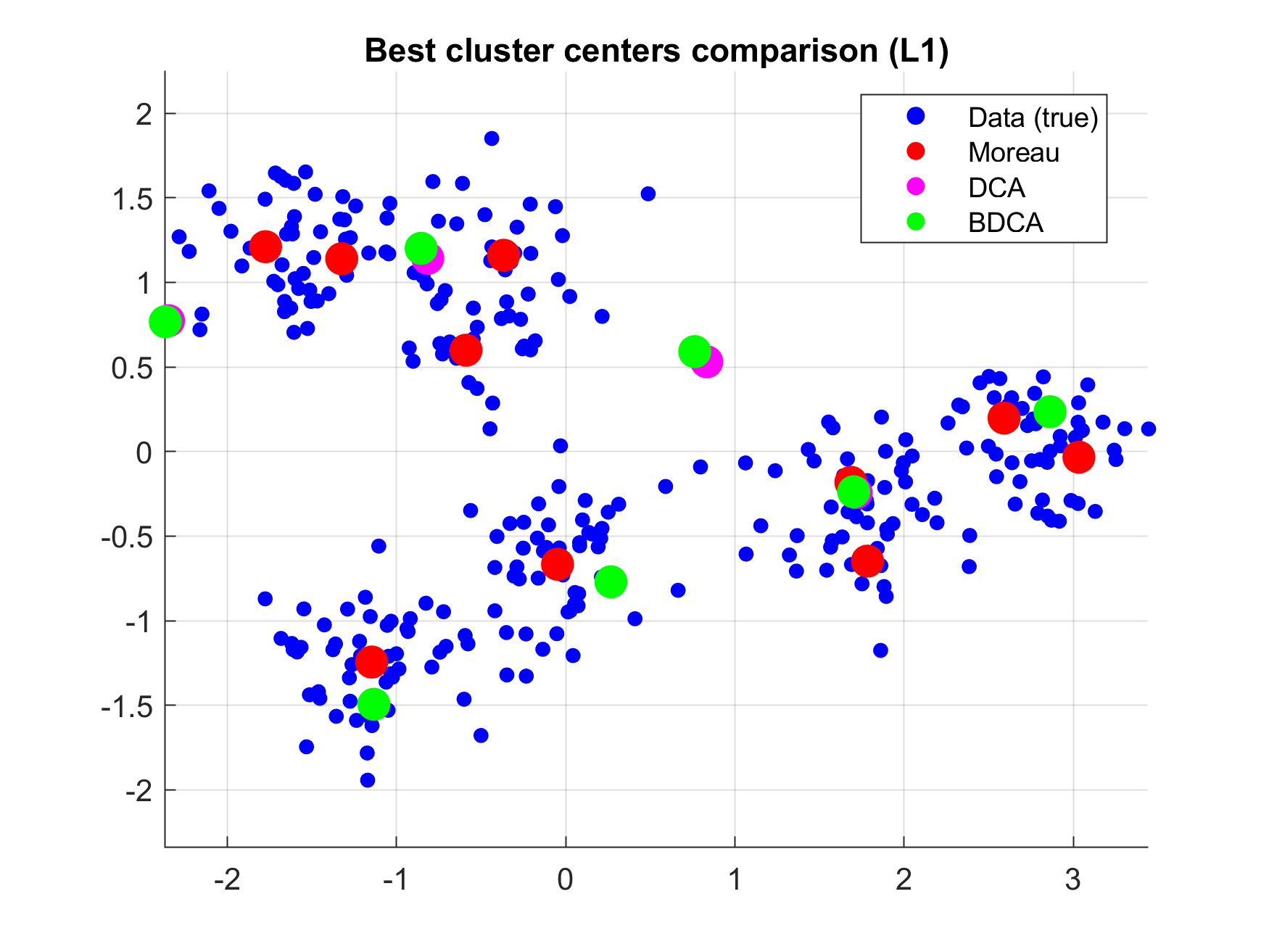}
\includegraphics[width=0.32\textwidth]{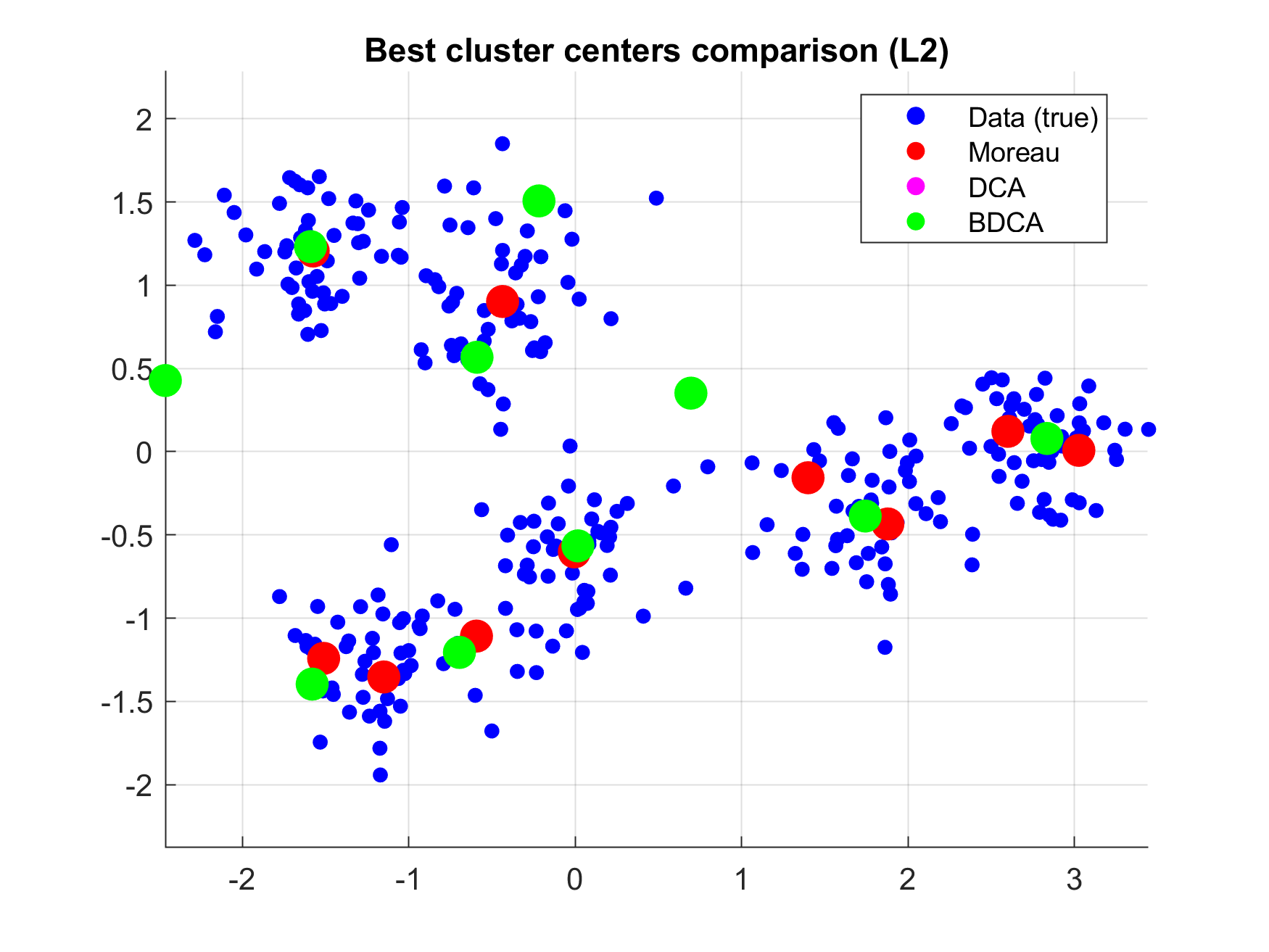}
\includegraphics[width=0.32\textwidth]{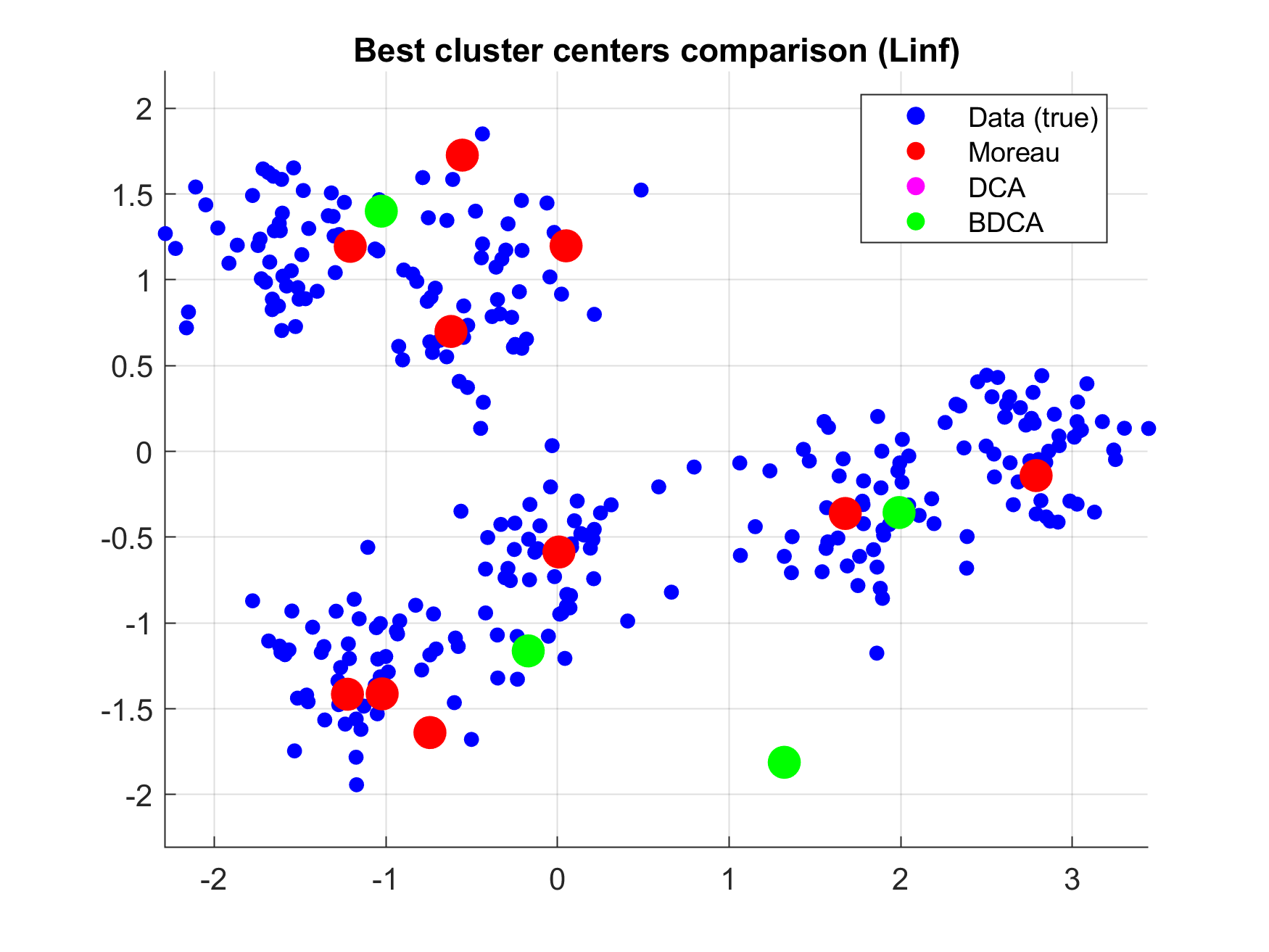}
\caption{Visualization of the best solution over 50 runs for the synthetic dataset under the fixed-center setting with $k_{\mathrm{init}} = 10$.}
\label{fig:fixed_k_10}
\end{figure}

\begin{figure}[H]
\centering
\includegraphics[width=0.32\textwidth]{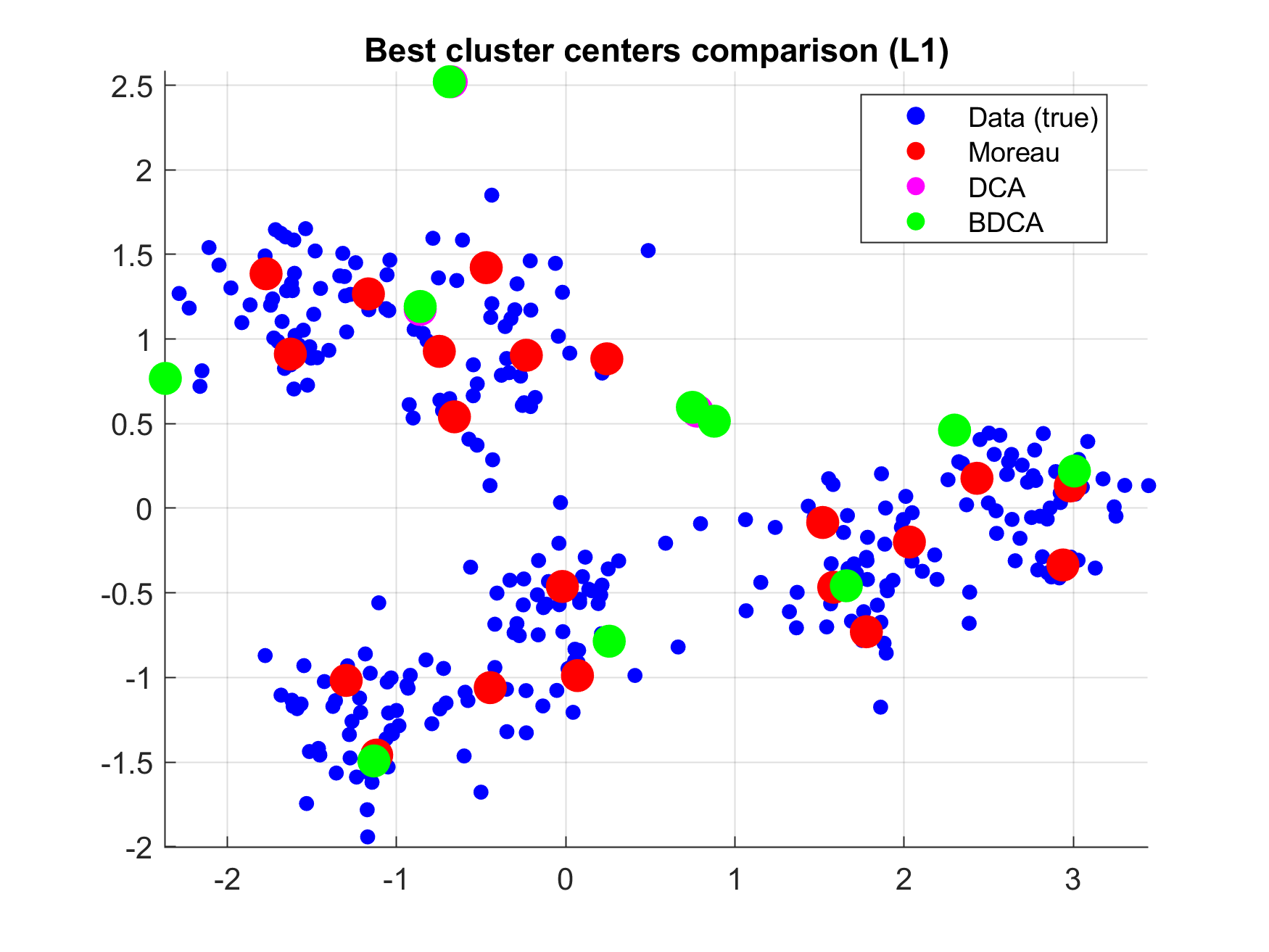}
\includegraphics[width=0.32\textwidth]{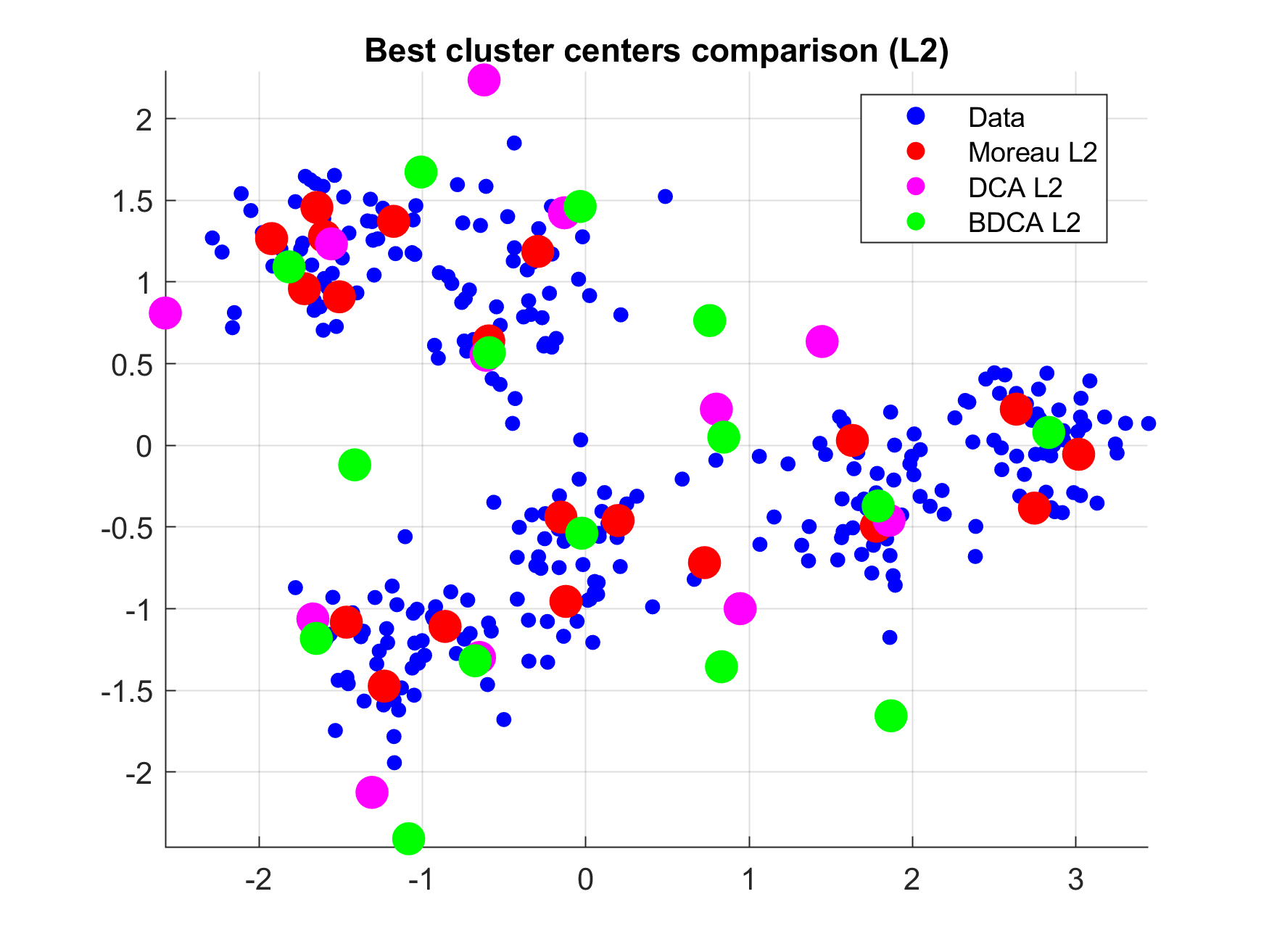}
\includegraphics[width=0.32\textwidth]{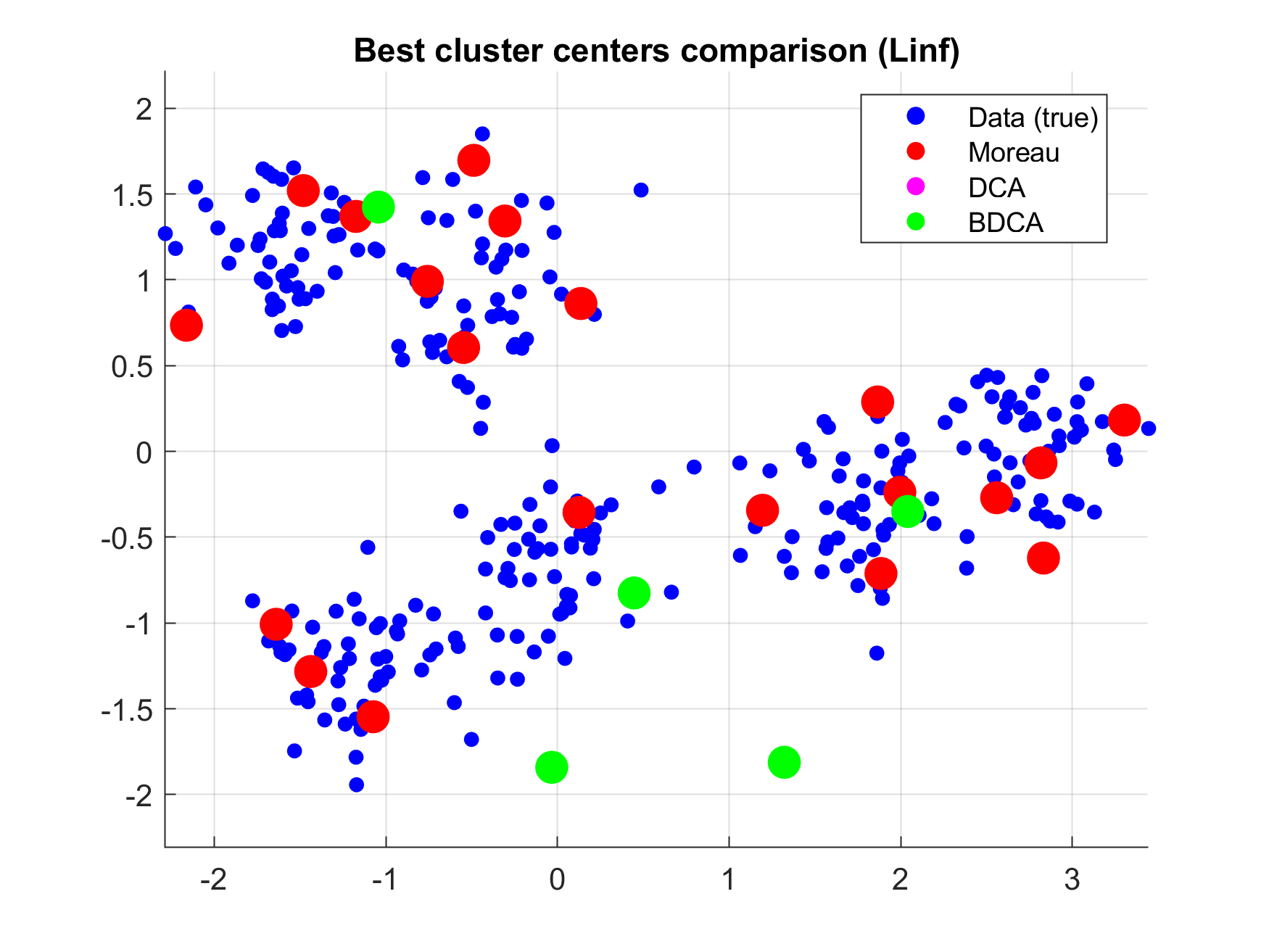}
\caption{Visualization of the best solution over 50 runs for the synthetic dataset under the fixed-center setting with $k_{\mathrm{init}} = 20$.}
\label{fig:fixed_k_20}
\end{figure}

Based on Tables~\ref{tab:fixed_k_10} and \ref{tab:fixed_k_20}, we observe that the proposed Moreau-based method generally outperforms or remains competitive with DCA and aBDCA across different datasets and norms. Specifically, Moreau achieves the best or highly competitive clustering accuracy while also attaining lower or comparable objective values and significantly faster running times. Its advantage is particularly evident under the $\ell_1$ and $\ell_\infty$ norms, where it consistently outperforms DCA and aBDCA, especially on the Synthetic and Glass datasets. These results indicate that the Moreau smoothing strategy effectively handles the nonsmooth objective and yields a more stable and efficient optimization process.

Another important observation concerns the number of clusters returned by each method. While Moreau tends to preserve a reasonable number of clusters close to the initialization, DCA and aBDCA frequently collapse to solutions with very few clusters (e.g., $k=2$ in several $\ell_\infty$ cases). This collapse leads to a poor representation of the data structure and explains the degradation in clustering accuracy. These observations are further supported by the visualizations in Figures~\ref{fig:fixed_k_10} and~\ref{fig:fixed_k_20}. The clusters obtained by Moreau are compact and well aligned with the underlying data distribution, whereas DCA and aBDCA often produce misplaced centers, with some lying far from the true cluster regions.

Overall, the results indicate that the proposed Moreau-based approach provides a more robust and reliable framework for clustering, particularly for nonsmooth formulations and challenging norms such as $\ell_\infty$.

\subsection{Adaptive number of centers}
\label{adaptive-synthetic-k}

In this setting, the number of centers is allowed to vary dynamically during the optimization process. The proposed adaptive Moreau-based algorithm extends the fixed-$k$ variant by incorporating two additional mechanisms: the removal of empty clusters and the merging of nearby centers based on a penalized objective. As a result, the method jointly optimizes both the locations of the centers and the number of active clusters. For completeness, the adaptive-$k$ scheme is summarized in Algorithm~\ref{alg:adaptive-moreau-gmw-2}.

\begin{algorithm}[H]
\caption{Moreau-Envelope Algorithm for GMWP with Adaptive Number of Centers}
\label{alg:adaptive-moreau-gmw-2}

\textbf{Input:} 
Demand points $A := \{a^1,\ldots,a^m\} \subset \mathbb{R}^n$, 
initial number of clusters $k_0$, 
smoothing parameter $\mu>0$, 
stepsize $\alpha>0$, 
tolerance $\varepsilon>0$, 
merge period $T_{\mathrm{merge}} \in \mathbb{N}$, 
merge quantile $q \in (0,1)$, 
penalty parameter $\lambda_k>0$.

\textbf{Initialization:} 
Choose $x^{(0)} \in \mathbb{R}^{n k_0}$, 
$t := 0$.

\begin{algorithmic}[1]

\Statex \textit{Let $k_t$ denote the number of clusters at iteration $t$.}

\Repeat

    \State Compute the gradient of the Moreau envelope:
    \[
    \nabla f_{F,k_t}^{\mu}(x^{(t)}) = 
    \frac{1}{\mu} \sum_{i=1}^m \Bigl( x^{(t)} - \mathrm{prox}_{\mu \varphi_i}(x^{(t)}) \Bigr),
    \]
    where $\mathrm{prox}_{\mu \varphi_i}(x^{(t)})$ modifies only the block 
    corresponding to the center nearest to $a^i$.

    \State Gradient update:
    \[
    x^{(t+\frac12)} = x^{(t)} - \alpha \nabla f_{F,k_t}^{\mu}(x^{(t)}).
    \]
    \State Reassign each demand point to its nearest center.

    \State Delete empty clusters. Let $J_t$ be the set of nonempty clusters and define
    \[
    x^{(t+1)} := \bigl( x_j^{(t+\frac12)} \bigr)_{j \in J_t}, 
    \qquad k_{t+1} := |J_t|.
    \]
    \State \textbf{Merge step (every $T_{\mathrm{merge}}$ iterations):}
    \If{$t \equiv 0 \ (\mathrm{mod}\ T_{\mathrm{merge}})$ and $k_{t+1} > 1$}
        \State Compute pairwise distances between centers and define a threshold
        \[
        \delta_t := \mathrm{quantile}_q\bigl(\{\|x_i^{(t+1)} - x_j^{(t+1)}\|\}_{i<j}\bigr)
        \]
        using the quantile $q$.

        \Repeat
            \State Find a pair $(i,j)$ such that $\|x_i^{(t+1)} - x_j^{(t+1)}\| \le \delta_t$.
            \State Let $C_i$ and $C_j$ be the sets of demand points assigned to clusters $i$ and $j$.
            \State Compute the merged center
            \[
            \bar{x}_{ij} := \arg\min_{x \in \mathbb{R}^n} \sum_{a^p \in C_i \cup C_j} \rho_F(x - a^p).
            \]
            \State Construct $\tilde{x}$ by replacing $(x_i^{(t+1)},x_j^{(t+1)})$ with $\bar{x}_{ij}$.
            \If{
            $
            f_F(\tilde{x}) + \lambda_k (k_{t+1}-1) \le f_F(x^{(t+1)}) + \lambda_k k_{t+1}
            $
            }
                \State Accept the merge: update $x^{(t+1)} \leftarrow \tilde{x}$ and $k_{t+1} \leftarrow k_{t+1}-1$.
                \State Reassign all demand points to the updated centers.
            \Else
                \State Reject the merge.
            \EndIf
        \Until{no further merge is accepted}

    \EndIf

    \State $t \leftarrow t + 1$.

\Until{
\(
\|\nabla f_{F,k_t}^{\mu}(x^{(t)})\| \le \varepsilon (\|x^{(t)}\| + 1)
\)
}

\textbf{Output:} Final number of clusters $k_t$ and cluster centers $x^{(t)}$.

\end{algorithmic}
\end{algorithm}
\begin{Remark}
\label{rem:merge_rule}
In the adaptive Moreau-envelop algorithm, the merging procedure is governed by
three key ingredients: the merge period $T_{\mathrm{merge}}$, the merge
threshold (defined via quantile $q$), and an objective-based acceptance test.

\begin{itemize}
    \item \textbf{Merge period $T_{\mathrm{merge}}$:} 
    The merge step is performed every $T_{\mathrm{merge}}$ iterations, not at every gradient update. 
    This allows cluster centers to stabilize locally before attempting any merge.

    \item \textbf{Merge threshold (quantile $q$):} 
    At each merge stage, only pairs of cluster centers that are geometrically close are considered for merging. 
    The threshold is set using a lower quantile of all pairwise center distances, ensuring scale invariance and preventing premature merging.

    \item \textbf{Acceptance test:} 
    A merge is accepted only if it reduces or keeps unchanged the penalized objective 
    $f_F(x) + \lambda_k k$, balancing data-fitting and model complexity.

    \item \textbf{Choice of merged center:} 
    Once a merge is accepted, the new center is selected according to the underlying norm:
    \begin{itemize}
        \item $\ell_1$: componentwise median of the points in the merged clusters,
        \item $\ell_2$: arithmetic mean or geometric median of the points,
        \item $\ell_\infty$: componentwise midrange (midpoint of min and max in each coordinate).
    \end{itemize}
    This ensures the merged center is representative under the corresponding norm.
\end{itemize}
\end{Remark}

We evaluate Algorithm~\ref{alg:adaptive-moreau-gmw-2} on the same datasets as in the fixed-$k$ setting. The initial number of centers is set to $k_{\mathrm{init}} = 10$ and $k_{\mathrm{init}} = 20$. For each setting, we perform 50 runs with randomly initialized centers.

For each method, the best solution over all runs is selected based on the objective value, and the corresponding performance metrics are reported.

The merge period is set to $T_{\mathrm{merge}} = 1$ and the merge threshold to $q = 10\%$. The cluster penalty parameter $\lambda_k$ is chosen empirically according to the complexity of each dataset. Specifically, $\lambda_k = 10$ for the synthetic dataset with 2 features, $\lambda_k = 25$ for datasets of moderate dimensionality (Iris with 4 features and Glass with 9 features), and $\lambda_k = 50$ for the Wine dataset with 13 features.

Numerical results are reported in Tables~\ref{tab:adaptive_k_10} and \ref{tab:adaptive_k_20} for $k_{\mathrm{init}} = 10$ and $k_{\mathrm{init}} = 20$, respectively. Corresponding visualizations of the best solutions on the synthetic dataset are presented in Figures~\ref{fig:adaptive_k_10} and \ref{fig:adaptive_k_20}.

\begin{table}[H]
\centering
\setlength{\tabcolsep}{3pt} 
\renewcommand{\arraystretch}{1}
\begin{tabular}{c|l|cccc|cccc|cccc}
\hline
& & \multicolumn{12}{c}{\textbf{Norm}} \\
\textbf{Dataset} & \textbf{Method}
& \multicolumn{4}{c}{$\ell_1$}
& \multicolumn{4}{c}{$\ell_2$}
& \multicolumn{4}{c}{$\ell_\infty$} \\
\cline{3-14}
& & ACC & Obj & Time & $k$ 
& ACC & Obj & Time & $k$ 
& ACC & Obj & Time & $k$ \\
\hline

\multirow{3}{*}{Synthetic}
& Moreau 
& \textbf{0.9867} & \textbf{139} & 2 & \textbf{6}
& 0.9833 & \textbf{111} & 2 & \textbf{6}
& \textbf{0.9700} & \textbf{101} & \textbf{7} & \textbf{6} \\

& DCA    
& 0.9767 & 188 & 4 & 7
& 0.9867 & 111 & 4 & 7
& 0.5000 & 206 & 53 & 3 \\

& aBDCA  
& 0.9767 & 184 & 1 & 7
& 0.9833 & 111 & 2 & 7
& 0.5000 & 206 & 60 & 3 \\

\hline
\multirow{3}{*}{Iris}
& Moreau 
& 0.8600 & \textbf{204} & 2 & \textbf{3}
& \textbf{0.8600} & \textbf{123.0} & 2 & 4
& \textbf{0.6600} & \textbf{124} & \textbf{0.4} & 2  \\

& DCA    
& 0.8667 & 217 & 7 & 4
& 0.8600 & 123.4 & 2 & 4
& 0.6600 & 127 & 29 & 2  \\

& aBDCA  
& 0.8200 & 217 & 2 & 4
& 0.8600 & 123.4 & 3 & 4
& 0.6600 & 127 & 102 & 2  \\

\hline
\multirow{3}{*}{Wine}
& Moreau 
& 0.9157 & \textbf{1193} & 2 & 5
& 0.9719 & 433 & \textbf{1} & 5
& 0.3989 & 336 & 0.3 & 1 \\

& DCA    
& 0.9607 & 1294 & 2 & 5
& 0.9719 & \textbf{426} & 7 & 6
& 0.6292 & \textbf{291} & 26 & 2 \\

& aBDCA  
& 0.9382 & 1302 & \textbf{1} & 6
& 0.9719 & \textbf{426} & 10 & 6
& 0.6292 & \textbf{291} & 37 & 2 \\
\hline
\multirow{3}{*}{Glass}
& Moreau 
& \textbf{0.5607} & \textbf{647} & 14 & 8
& 0.5701 & 310 & 15 & 8 
& \textbf{0.5327} & \textbf{231} & \textbf{4} & 5  \\

& DCA    
& 0.4813 & 1044 & 7 & 5  
& 0.5701 & \textbf{300} & 11 & 8
& 0.4252 & 307 & 45 & 2  \\

& aBDCA  
& 0.4813 & 1042 & 2 & 5  
& 0.5654 & 304 & 11 & 8 
& 0.4252 & 307 & 56 & 2  \\
\hline
\end{tabular}
\caption{Comparison of methods for the adaptive-center setting with $k_{\mathrm{init}} = 10$.}
\label{tab:adaptive_k_10}
\end{table}

\begin{table}[H]
\centering
\setlength{\tabcolsep}{3pt} 
\renewcommand{\arraystretch}{1}
\begin{tabular}{c|l|cccc|cccc|cccc}
\hline
& & \multicolumn{12}{c}{\textbf{Norm}} \\
\textbf{Dataset} & \textbf{Method}
& \multicolumn{4}{c}{$\ell_1$}
& \multicolumn{4}{c}{$\ell_2$}
& \multicolumn{4}{c}{$\ell_\infty$} \\
\cline{3-14}
& & ACC & Obj & Time & $k$ 
& ACC & Obj & Time & $k$ 
& ACC & Obj & Time & $k$ \\
\hline

\multirow{3}{*}{Synthetic}
& Moreau 
& \textbf{0.9867} & \textbf{139} & \textbf{1} & \textbf{6}
& 0.9833 & 111 & 3 & \textbf{6}
& \textbf{0.9700} & \textbf{101} & \textbf{4} & \textbf{6}  \\

& DCA    
& 0.8267 & 193 & 8 & 6
& 0.9833 & 111 & 14 & 8
& 0.5000 & 206 & 45 & 3  \\

& aBDCA  
& 0.9733 & 185 & 1 & 7
& 0.9867 & 111 & 1 & 7
& 0.5000 & 206 & 54 & 3  \\
\hline

\multirow{3}{*}{Iris}
& Moreau 
& 0.8600 & \textbf{204} & 2 & \textbf{3}
& \textbf{0.8600} & \textbf{123.0} & \textbf{3} & 4
& 0.6600 & \textbf{124} & \textbf{0.3} & 2 \\

& DCA    
& 0.8667 & 217 & 7 & 4
& 0.8600 & 123.4 & 9 & 4
& 0.6600 & 127 & 12 & 2 \\

& aBDCA  
& 0.8200 & 217 & 2 & 4
& 0.8467 & 123.5 & 7 & 4
& 0.6600 & 127 & 16 & 2 \\
\hline

\multirow{3}{*}{Wine}
& Moreau 
& 0.9607 & \textbf{1211} & 2 & 5
& \textbf{0.9775} & 431 & \textbf{1} & 5
& 0.399 & 336 & 0.3 & 1 \\

& DCA    
& 0.9551 & 1294 & 4 & 6
& 0.9719 & \textbf{424} & 15 & 6
& 0.629 & 291 & 41 & 2 \\

& aBDCA  
& 0.9719 & 1297 & 1 & 6
& 0.9719 & \textbf{424} & 24 & 6
& 0.629 & 291 & 58 & 2 \\
\hline

\multirow{3}{*}{Glass}
& Moreau 
& \textbf{0.5701} & \textbf{646} & 23 & 7  
& 0.5748 & 302 & \textbf{17} & 8
& \textbf{0.5000} & \textbf{230} & \textbf{6} & 7  \\

& DCA    
& 0.4860 & 1028 & 17 & 5  
& 0.5794 & \textbf{293} & 27 & 9
& 0.4252 & 307 & 79 & 2  \\

& aBDCA  
& 0.4907 & 1020 & 4 & 5  
& 0.5794 & \textbf{293} & 20 & 9
& 0.4252 & 307 & 97 & 2  \\
\hline

\end{tabular}
\caption{Comparison of methods for the adaptive-center setting with $k_{\mathrm{init}} = 20$.}
\label{tab:adaptive_k_20}
\end{table}

\begin{figure}[H]
\centering
\includegraphics[width=0.32\textwidth]{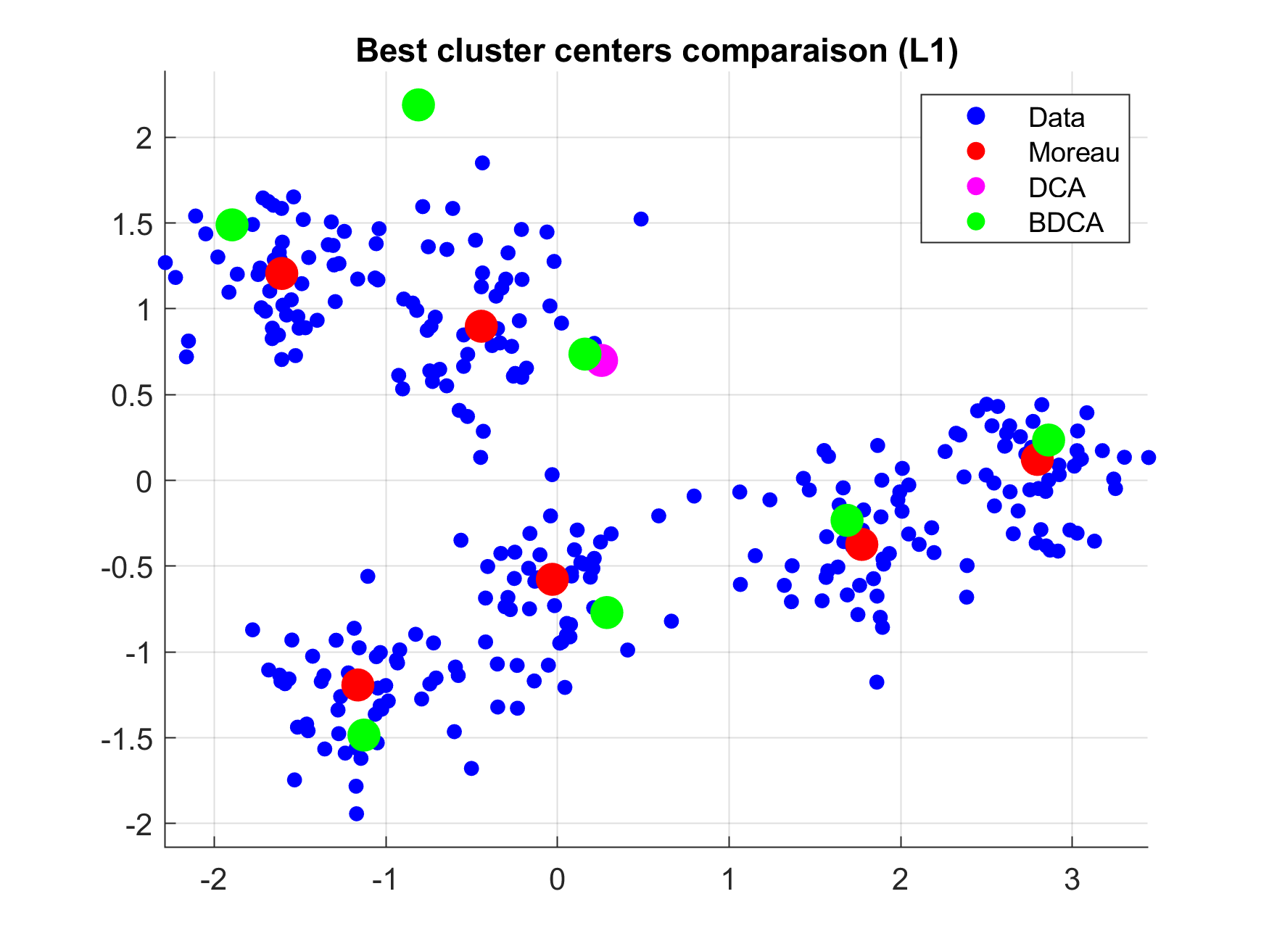}
\includegraphics[width=0.32\textwidth]{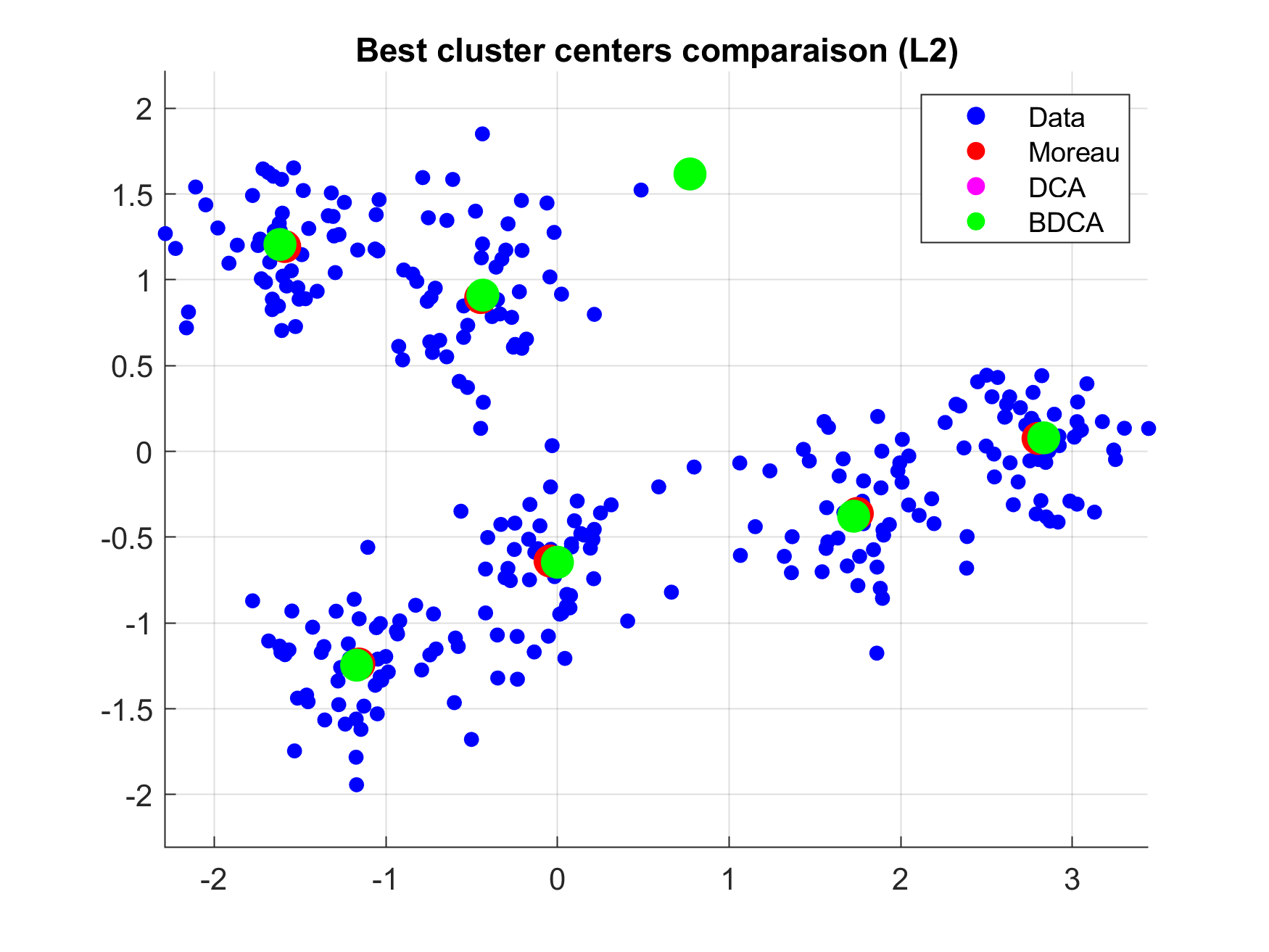}
\includegraphics[width=0.32\textwidth]{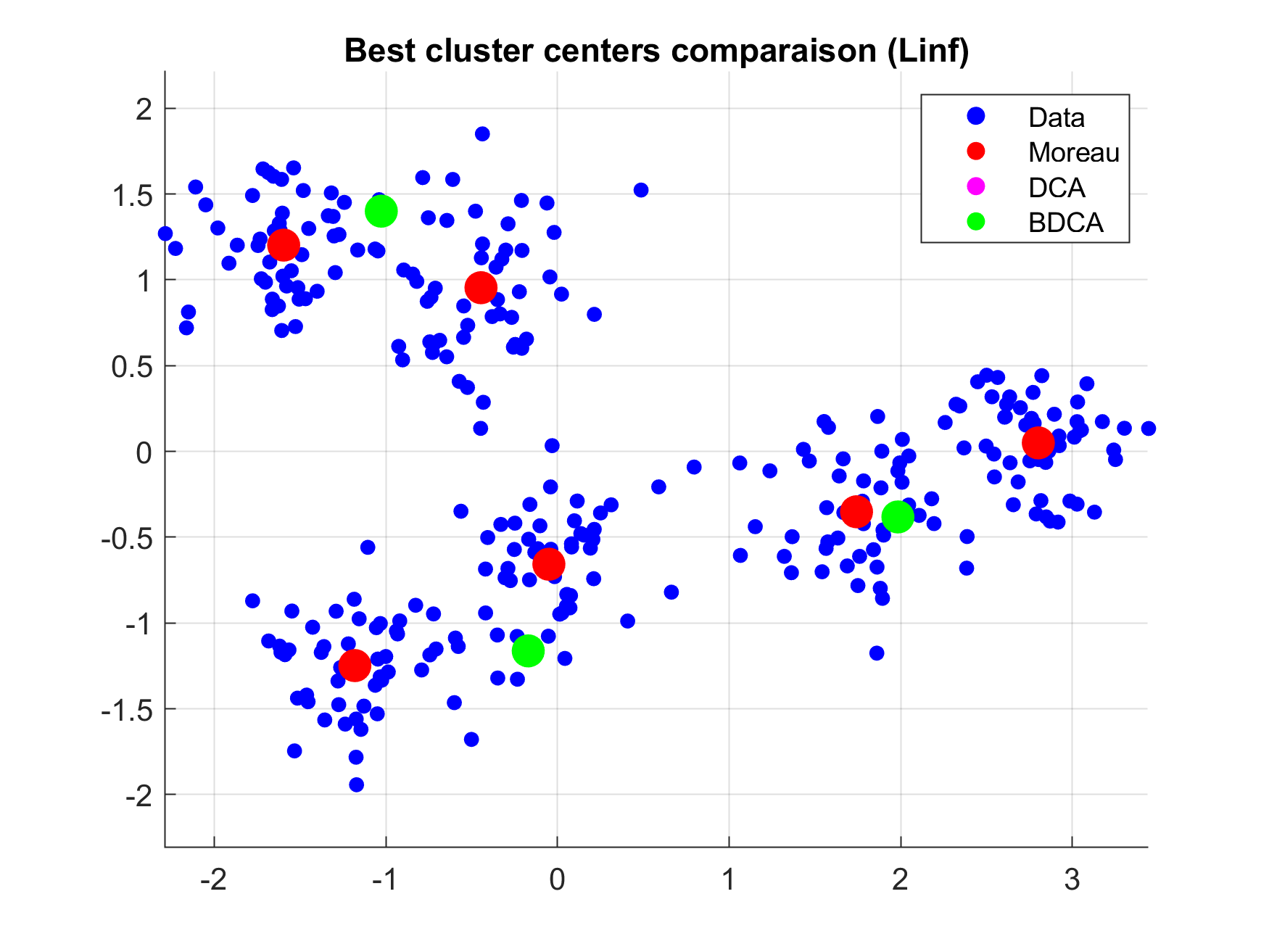}
\caption{Visualization of the best solution for the synthetic dataset under the adaptived-center setting with $k_{\mathrm{init}} = 10$.}
\label{fig:adaptive_k_10}
\end{figure}

\begin{figure}[H]
\centering
\includegraphics[width=0.32\textwidth]{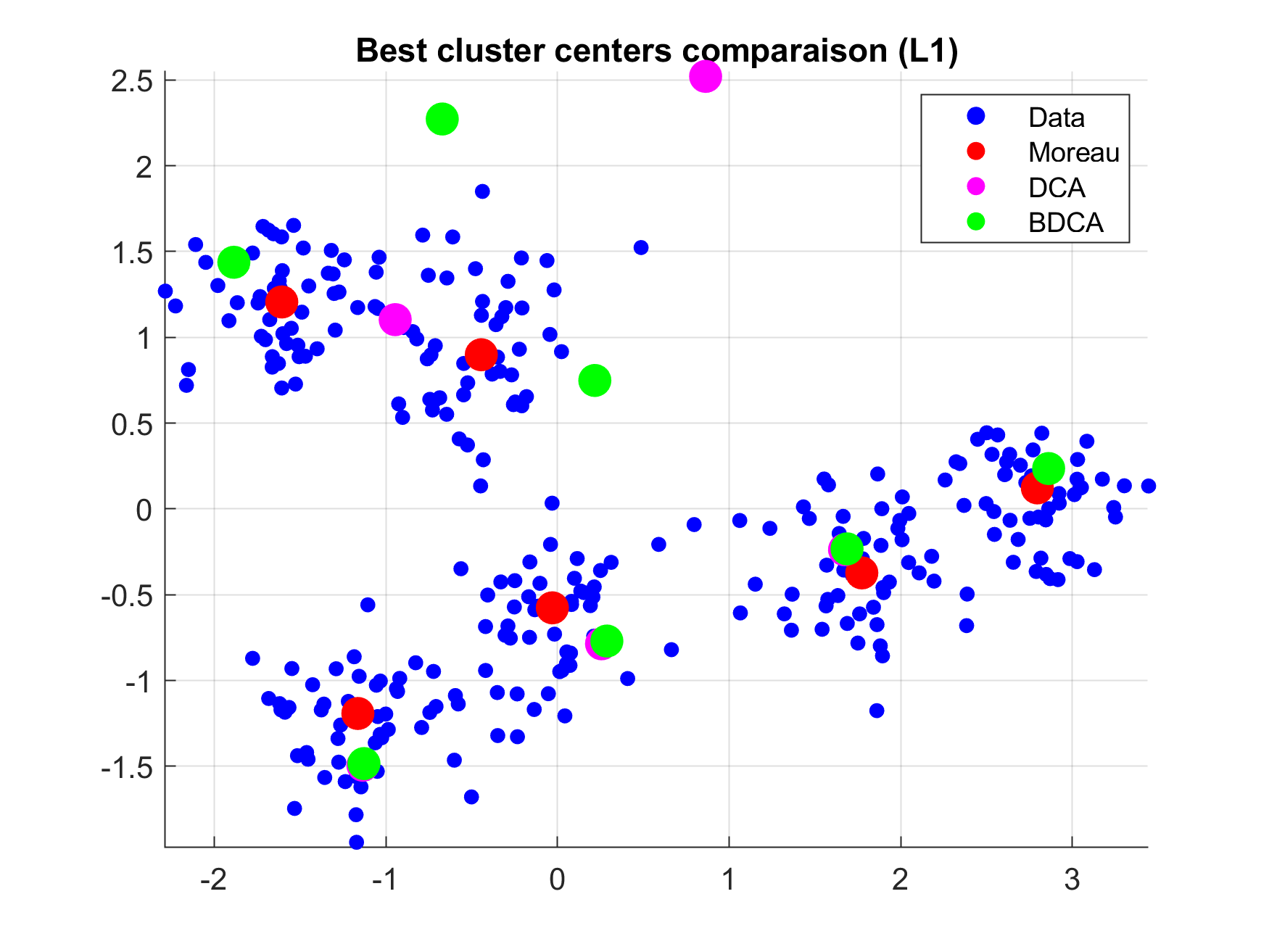}
\includegraphics[width=0.32\textwidth]{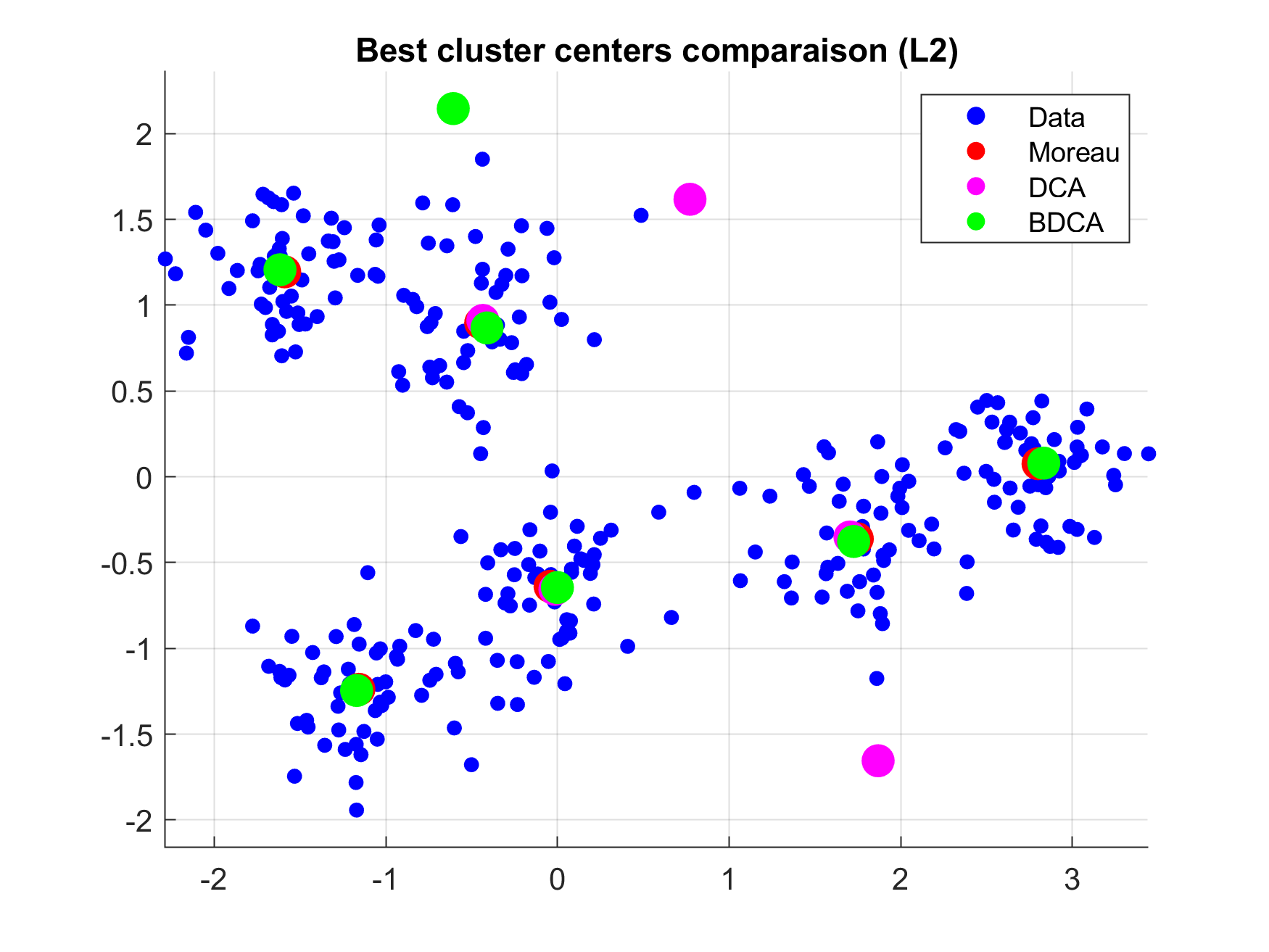}
\includegraphics[width=0.32\textwidth]{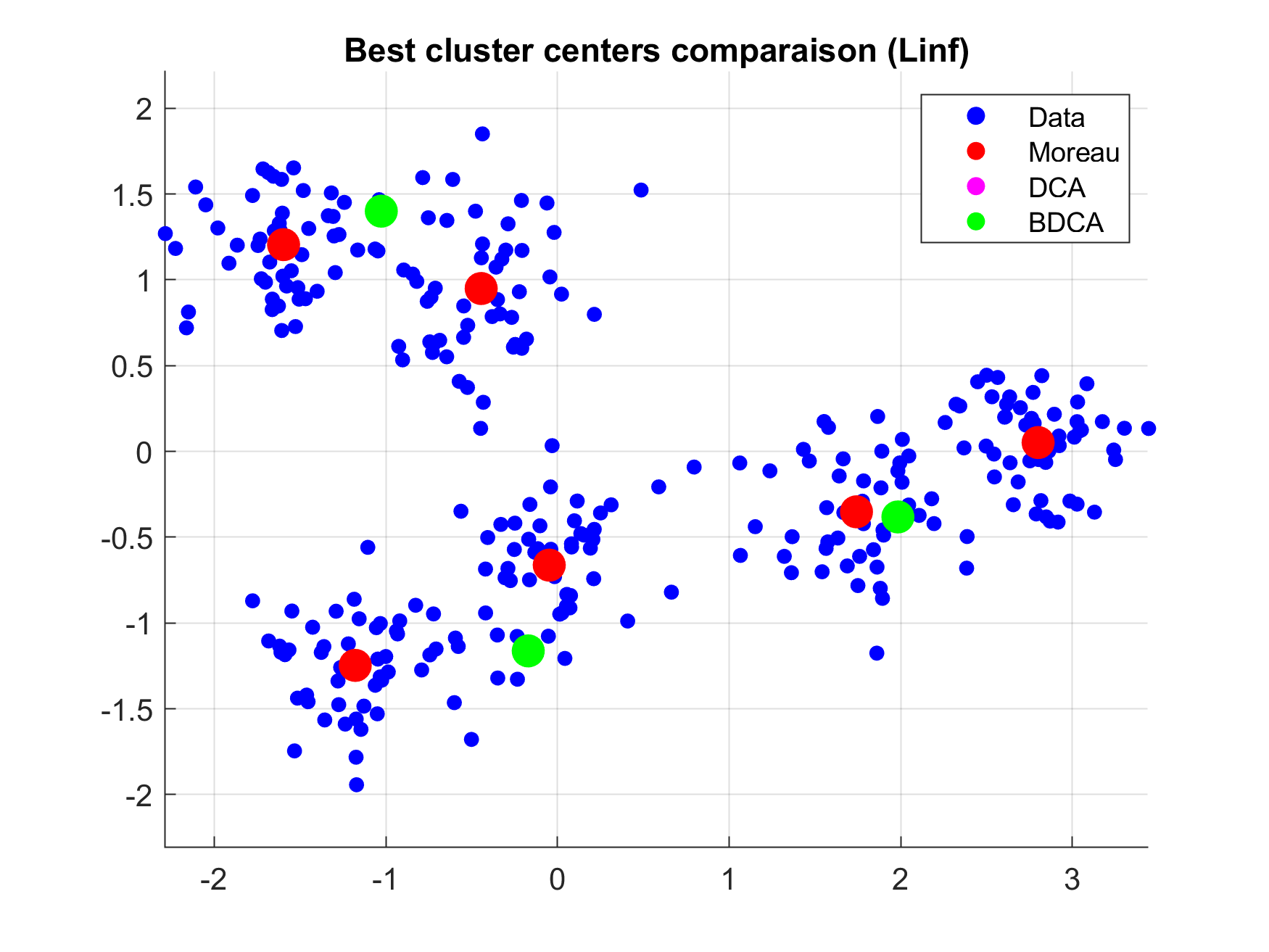}
\caption{Visualization of the best solution for the synthetic dataset under the adaptive-center setting with $k_{\mathrm{init}} = 20$.}
\label{fig:adaptive_k_20}
\end{figure}

Based on Tables~\ref{tab:adaptive_k_10} and \ref{tab:adaptive_k_20}, the proposed Moreau-based method generally outperforms DCA and aBDCA across various datasets and norms. In particular, it often achieves lower objective values while maintaining comparable or higher clustering accuracy. The advantage is especially pronounced under the $\ell_1$ and $\ell_\infty$ norms, where the Moreau-based approach consistently surpasses the other methods on datasets such as Synthetic and Glass.

For the Synthetic dataset, which has a well-separated cluster structure, the Moreau method is able to recover the true number of clusters and achieves near-perfect accuracy. This indicates that the proposed smoothing strategy preserves the intrinsic geometry of the data while effectively handling nonsmoothness. In contrast, DCA and aBDCA tend to merge clusters under the $\ell_\infty$ norm, resulting in a significantly reduced number of clusters (e.g., $k=3$) and a substantial drop in accuracy.

Another notable aspect is the number of clusters returned by each method. Moreau typically produces a reasonable and stable estimate of $k$, whereas DCA and aBDCA sometimes collapse to very small values, particularly under the $\ell_\infty$ norm. This behavior leads to degraded clustering quality and suggests a lack of robustness when dealing with nonsmooth formulations.

These observations are further supported by the visual results in Figures~\ref{fig:adaptive_k_10} and~\ref{fig:adaptive_k_20}. The clusters obtained by the Moreau-based method are more compact and better aligned with the underlying data distribution, while DCA and aBDCA occasionally produce poorly located centers or merged clusters.

Overall, the results demonstrate that the Moreau-based approach is more robust and reliable, especially for nonsmooth clustering problems and challenging norms such as $\ell_\infty$.

\section{Conclusion}\label{sec:Conclusion}

In this paper, we proposed an efficient clustering framework based on the
Moreau envelope smoothing technique for nonsmooth and nonconvex optimization
problems with an unknown number of clusters.
By smoothing the nonsmooth components of the objective function,
the original clustering model is transformed into a structured problem
that can be effectively solved using first-order methods.

The resulting algorithm starts from an initial overestimation of the number
of clusters and automatically reduces it during the optimization process
by eliminating empty clusters and merging redundant ones.
This adaptive mechanism allows the method to balance clustering accuracy
and model complexity without requiring the number of clusters to be specified
in advance.

Numerical experiments on both synthetic and real datasets demonstrate that
the proposed approach is competitive with existing methods such as DCA and aBDCA.
In particular, the adaptive Moreau-based algorithm achieves comparable or
better objective values while significantly reducing the number of clusters,
with reasonable computational cost.
These results highlight the effectiveness and flexibility of the proposed
method for large-scale clustering problems.

\section*{Declarations}

{\bf Competing Interests} There are no conflicts of interest or competing interests related to this manuscript.

\end{document}